\documentclass{amsart}
\usepackage{amssymb,amsmath, graphicx,enumerate, setspace}
\usepackage{amsthm, mathtools}
\usepackage{supertabular}

\newcommand{\rr}{\mathbb{R}}
\newcommand{\be}{\begin{enumerate}}
\newcommand{\ee}{\end{enumerate}}
\newcommand{\beq}{\begin{equation}}
\newcommand{\eeq}{\end{equation}}
\newcommand{\beqs}{\begin{equation*}}
\newcommand{\eeqs}{\end{equation*}}
\newcommand{\bea}{\begin{eqnarray}}
\newcommand{\eea}{\end{eqnarray}}
\newcommand{\beas}{\begin{eqnarray*}}
\newcommand{\eeas}{\end{eqnarray*}}
\newcommand{\st}{\,:\,}

\newcommand{\fl}{\par\noindent}
\def\({\left(}
\def\){\right)}

\swapnumbers
\theoremstyle{plain}
\newtheorem{theorem}{Theorem}
\newtheorem{proposition}[theorem]{Proposition}
\newtheorem{lemma}[theorem]{Lemma}
\newtheorem{corollary}[theorem]{Corollary}

\newtheorem{fact}[theorem]{Fact}

\theoremstyle{definition}

\newtheorem{notation}[theorem]{Notation}

\theoremstyle{remark}
\newtheorem*{remark}{Remark}

\theoremstyle{property}

\begin{document}
\title[Dvoretzky--Kiefer--Wolfowitz Inequalities for the Two-sample Case]{Dvoretzky--Kiefer--Wolfowitz Inequalities for the Two-sample Case}
\date{\today}
\author{Fan Wei}
\address[Fan Wei]{M.I.T.}
\email[Fan Wei]{fan\textunderscore{}wei@mit.edu}
\author{R. M. Dudley}
\address[R. M. Dudley]{M.I.T. Mathematics Department}
\email[R. M. Dudley]{rmd@math.mit.edu}
\begin{abstract}
The Dvoretzky--Kiefer--Wolfowitz (DKW) inequality says that
if $F_n$ is an empirical distribution function for variables
i.i.d.\ with a distribution function $F$, and $K_n$ is the
Kolmogorov statistic $\sqrt{n}\sup_x|(F_n-F)(x)|$, then there
is a finite constant $C$ such that for any $M>0$,
$\Pr(K_n>M) \leq C\exp(-2M^2).$ Massart proved that
one can take $C=2$ (DKWM inequality) which is sharp for $F$
continuous. We consider
the analogous Kolmogorov--Smirnov statistic $KS_{m,n}$ for the 
two-sample case and show that for $m=n$, the DKW inequality holds
with $C=2$ if and only if $n\geq 458$. For $n_0\leq n<458$ it holds
for some $C>2$ depending on $n_0$.

For $m\neq n$, the DKWM inequality fails for the three pairs
$(m,n)$ with $1\leq m < n\leq 3$.  We found by computer search that
for $n\geq 4$, the DKWM inequality always holds for $1\leq m< n\leq 200$,
and further that it holds for $n=2m$ with $101\leq m\leq 300$.
We conjecture that the DKWM inequality holds for pairs $m\leq n$ with 
the $457+3 =460$ exceptions mentioned. 
\end{abstract} 
%
%

\maketitle

%
%

\section{Introduction}
This paper is a long version, giving many more details, of
our shorter paper \cite{FWRDshort}.
Let $F_n$ be the empirical distribution function based on an
i.i.d.\ sample from a distribution function $F$, let
$$D_n := \sup_x|(F_n-F)(x)|,$$ and let $K_n$ be the Kolmogorov
statistic $\sqrt{n}D_n$.
Dvoretzky, Kiefer, and Wolfowitz in 1956 \cite{DKW} proved that 
 there is a finite constant $C$ such that for all $n$ and all $M>0$,
\beq\label{DKWineq} 
\text{Pr}(K_n \geq M ) \leq C \exp(-2M^2 ). 
\eeq
We call this the DKW inequality.
Massart in 1990 \cite{Massart2}
proved (\ref{DKWineq}) with the sharp constant $C=2$, 
which we will call the DKWM inequality.
In this paper we consider possible extensions of these inequalities
to the two-sample case, as follows. For $1\leq m \leq n$,
the null hypothesis $H_0$ is that $F_m$ and
$G_n$ are independent empirical distribution functions from a
continuous distribution function $F$, based altogether on $m+n$ samples
i.i.d.\ $(F)$. Consider the Kolmogorov--Smirnov statistics
\beq\label{KSstats} 
D_{m,n} = \text{sup}_x\mid (F_m-G_n)(x) \mid, \ \ 
KS_{m,n} = \sqrt{\dfrac{mn}{m+n}}D_{m,n}.
\eeq 
 All probabilities to be considered are under $H_0$.

For given $m$ and $n$ let $L=L_{m,n}$ be their least common multiple.
Then the possible values of $D_{m,n}$ are included in the set of
all $k/L$ for $k=1,\dotsc,L$. If $n=m$ then all these values are
possible. The possible values of $KS_{m,n}$ are thus of the form
\beq\label{possM}
M=\sqrt{(mn)/(m+n)}k/L_{m,n}.
\eeq
 We will say that the DKW (resp.\ DKWM) inequality
holds in the two-sample case for given $m,n,$ and $C$ 
(resp.\ $C=2$) if for all $M>0$, the following holds: 
\beq\label{DKWksineq} 
P_{m,n,M}:=\text{Pr}(KS_{m,n} \geq M ) \leq C \exp(-2M^2 ).
\eeq

It is well known that as $m\to+\infty$ and $n\to+\infty$, for
any $M>0$,
\beq\label{brbrlim}
P_{m,n,M}\ \to\ \beta(M) := \text{Pr}(\sup_{0\leq t\leq 1}|B_t|>M)
= 2\sum_{j=1}^{\infty}(-1)^{j-1}\exp(-2j^2M^2),
\eeq
where $B_t$ is the Brownian bridge process.
\begin{remark} For $M$ large enough so that $H_0$ can be rejected
according to the asymptotic distribution given in (\ref{brbrlim})
at level $\alpha\leq 0.05$, the series in (\ref{brbrlim}) is very
close in value to its first term $2\exp(-2M^2)$, which is the
DKWM bound (when it holds). Take $M_{\alpha}$ such that
$2\exp(-2M_{\alpha}^2)=\alpha$, then for example we will have
$\beta(M_{.05})\doteq 0.04999922$, 
$\beta(M_{.01})\doteq 0.009999999$. 
\end{remark}

Let $r_{\max}=r_{\max}(m,n)$ be the largest ratio 
$P_{m,n,M}/(2\exp(-2M^2))$ over all possible values of $M$
for the given $m$ and $n$.
We summarize our main findings in Theorem \ref{meqn} and  Facts
 \ref{mneqn}, \ref{doubles}, and \ref{lowbdreler}. 
\begin{theorem}\label{meqn}
 For $m=n$ in the two-sample case:
\begin{itemize}
\item[$(a)$] The DKW inequality always holds with $C=e\doteq 2.71828.$
\item[$(b)$] For  $m=n\geq 4,$ the smallest $n$ such that 
$H_0$ can be rejected at
level $0.05$, the DKW inequality holds with $C=  2.16863$.
\item[$(c)$] The DKWM inequality holds for all $m=n\geq 458$, i.e.,
for all $M>0$,
\beq 
P_{n,n,M}=\text{Pr}\left(KS_{n,n} \geq M  \right)
\leq 2 e^{-2M^2}. \label{mt}
\eeq
\item[$(d)$]  For each $m=n<458$, the DKWM inequality fails for some $M$ given by (\ref{possM}).
\item[$(e)$]  For each $m=n<458$, the DKW inequality holds for 
$C=2(1+\delta_n)$ for some $\delta_n>0$, where for $12\leq n\leq
457$,
\beq 
\delta_n < -\dfrac{0.07}{ n} + \dfrac{40}{n^2} -\dfrac{400}{n^3}. \label{dnubo}
\eeq 
\end{itemize}
 \end{theorem}
\begin{remark} The bound on the right side of (\ref{dnubo}) is larger than
$2\delta_n$ for $n = 16,$ $40,70,$ $440,$ and $445$
for example, but is less than $1.5\delta_n$ for $125\leq n\leq 415$.
 It is less than $1.1\delta_n$ for $n=285,\ 325,\ 345$.
\end{remark}
Theorem \ref{meqn} (a), (b), and (c) are proved in Section \ref{meqnsec}.
Parts (d) and (e), and also parts (a) through (c) for $n<6395$, were
found by computation.
 
For $m\neq n$ we have no general or theoretical proofs but report
on computed values. The methods of computation are summarized in
Subsection \ref{computation}. Detailed results in support of the
following three facts are given in Subsection \ref{Facts2-3} and
Appendix B.
\begin{fact}\label{mneqn}
 Let $1\leq m < n\leq 200$. Then:
\begin{itemize}
\item[$(a)$]  For $n\geq 4$, the DKWM inequality holds.
\item[$(b)$]  For each $(m,n)$ with $1\leq m<n\leq 3$, the DKWM
inequality fails, in the case of
$\Pr(D_{m,n}\geq 1)$.
\item[$(c)$]  For $3\leq m\leq 100$, the $n$ with $m<n\leq 200$ having largest
$r_{\max}$ is always $n=2m$.
\item[$(d)$]  For $102 \leq m\leq 132$ and $m$ even, the largest $r_{\max}$ is
always found for $n=3m/2$ and is increasing in $m$.
\item[$(e)$]  
For $169\leq m\leq 199$ and $m<n\leq 200$, 
the largest $r_{\max}$ occurs for $n=m+1$.
\item[$(f)$]  For $m=1$ and $4\leq n\leq 200$, the largest $r_{\max}=0.990606$ 
occurs for $n=4$ and $d=1$.
For $m=2$ and $4\leq n\leq 200$, the largest $r_{\max}= 0.959461$ 
occurs for $n=4$ and $d=1$.
\end{itemize}
\end{fact}

In light of Fact \ref{mneqn}(c) we further found:
\begin{fact}\label{doubles}
For $n=2m$:
\begin{itemize}
\item [$(a)$] For $3\leq m\leq 300$, the DKWM inequality holds;
$r_{\max}(m,2m)$ has relative minima at $m=6,\ 10$, and $16$ but is
increasing for $m\geq 16$, up to $0.9830$ at $m=300$. 
\item [$(b)$] The $p$-values forming the numerators of $r_{\max}$ for
$100\leq m\leq 300$ are largest for $m=103$ 
where $p\doteq 0.3019$ and smallest at $m=294$ where
$p\doteq 0.2189$.
\item[$(c)$]  For $101\leq m\leq 199$, the smallest $r_{\max}$ for $n=2m$,
namely $r_{\max}(101,202) \doteq 0.97334$, is larger than
every $r_{\max}(m',n')$ for $101\leq m'<n'\leq 200$, all of which are 
less than $0.95$, the largest being $r_{\max}(132,198) \doteq 0.9496$.
\item [$(d)$] For $3\leq m\leq 300$, $r_{\max}$  is attained at 
$d_{\max}={k_{\max}}/{n}$
which is decreasing in $n$ when $k_{\max}$ is constant but
jumps upward when $k_{\max}$ does; $k_{\max}$ is nondecreasing in $m$.
\end{itemize}
\end{fact}
The next fact shows that for a wide range of pairs $(m,n)$, but
not including any with $n=m$ or $n=2m$, the correct $p$-value
$P_{m,n,M}$ is substantially less than its upper bound $2\exp(-2M^2)$
and in cases of possible significance at the $0.05$ level or less,
likewise less than the asymptotic $p$-value $\beta(M)$:
\begin{fact}\label{lowbdreler} Let  $100 < m < n\leq 200$. Then:
\begin{itemize}
\item [$(a)$] The ratio $2\exp(-2M^2)/P_{m,n,M}$ 
is always at least $1.05$ for all possible values of $M$ in $(\ref{possM})$. 
The same is true if the numerator is replaced by the asymptotic probability
$\beta(M)$ and $\beta(M) \leq 0.05$.
\item [$(b)$] If in addition $m=101$, $103$, $107$, $109$, or $113$, then
part $(a)$ holds with $1.05$ replaced by $1.09$.
\end{itemize} 
\end{fact}
\begin{remark}
We found that in some ranges $d_0(m,n)\leq D_{m,n}\leq 1/2$,
too few significant digits of small $p$-values (less than $10^{-14}$) 
could be computed
by the method we used for $0<D_{m,n}<d_0(m,n)$. But, one can 
compute accurately an upper bound for such $p$-values, 
which we used to verify Facts \ref{mneqn}, \ref{doubles},
and \ref{lowbdreler}
for those ranges. We give details in Section \ref{mneqnsec} and
Appendix B.  

We have in the numerator of $r_{\max}$ the $p$-values of $0.2189$ 
(corresponding to $m=294$) or more  
in Fact \ref{doubles}(b)
(Table \ref{tab:D2}), and similarly $p$-values of $0.26$ or more
in Table \ref{tab:D1} and $0.27$ or more in Table \ref{tab:HTH}.
These substantial $p$-values suggest, although they of course do not
prove, that more generally, large $r_{\max}$ do not tend to occur at small
$p$-values.

\end{remark}

\section{Proof of  Theorem \ref{meqn}}\label{meqnsec}

B.\ V.\ Gnedenko and V.\ S.\ Korolyuk in 1952 \cite{GK} 
gave an explicit formula for 
$P_{n,n,M}$, and M.\ Dwass (1967) \cite{Dwass} gave another proof.
The technique is older: the reflection principle dates back to
Andr\'e \cite{Andre}.  Bachelier in 1901
\cite[pp.\ 189-190]{Bachelier01} is the earliest reference we could
find for the method of repeated reflections, applied to symmetric
random walk. He emphasized that the formula there is rigorous
(``rigoureusement exacte''). Expositions in several later
books we have seen, e.g.\ in 1939
\cite[p.\ 32]{Bachelier39}, are not so rigorous, assuming a
normal approximation and thus treating repeated reflections of
Brownian motion.
According to J.\ Blackman
\cite[p.\ 515]{Blackman56} the null distribution of $\sup|F_n-G_n|$ had
in effect ``been treated extensively by Bachelier'' in 1912, \cite
{Bachelier12} ``in connection with certain gamblers'-ruin problems.''
  
The formula is given in 
the following proposition.
\begin{proposition}[Gnedenko and Korolyuk]\label{realprobability}
If $M = k/\sqrt{2n}$, where $1 \leq k \leq n$ is an 
integer, then
\beqs 
\text{Pr}\left( KS_{n,n} 
\geq M  \right) =\dfrac{2}{\binom{2n}{n}} 
\left(\sum_{i=1}^{\lfloor{{n}/{k}}\rfloor} (-1)^{i-1}
\binom{2n}{n+ i k}    \right).
\eeqs
\end{proposition}

Since the probability 
$P_{n,n,M}=\text{Pr}\left(KS_{n,n}  \geq M  \right)$ 
is clear\-ly not greater than 1, we just need to consider the $M$ such that 
$$
2e^{-2M^2} \leq 1,
$$ 
i.e., we just need to consider the integer pairs $(n,k)$ where 
\beq 
k \geq \sqrt{n\ln{2}}. \label{ksmall} 
\eeq

The exact formula for $P_{n,n,M}$ is complicated. Thus we want to determine upper 
bounds for $P_{n,n,M}$ which are of simpler forms. We prove the main 
theorem by two steps: we first find two such upper bounds for $P_{n,n,M}$ as in Lemma \ref{up1} and \ref{up2} and 
then show (\ref{mt}) holds when $P_{n,n,M}$ is replaced by the two upper bounds for 
two ranges of pairs $(k,n)$ respectively, as will be stated in 
Propositions
\ref{thm1} and \ref{thm2}. 

\begin{lemma}   \label{up1}
An upper bound for $P_{n,n,M}$ can be given by 
$
{2\binom{2n}{n+k}}/{\binom{2n}{n}}.$
\end{lemma}

\begin{proof}
This is clear from Proposition \ref{realprobability}, since the summands 
alternate in signs and decrease in magnitude. Therefore we must have 
$$ 
\sum_{i=2}^{\lfloor{{n}/{k}}\rfloor} (-1)^{i-1}\binom{2n}{n+ i k} \leq 0.$$ \end{proof}

As a consequence of Lemma \ref{up1}, to prove (\ref{mt}) for a pair $(n,k)$, 
it will suffice to show that
\beq
 {2\binom{2n}{n+k}}/{\binom{2n}{n}} < 2\exp\left({-{k^2}/{n}}\right).
\label{simple1} 
\eeq 

We first define some auxiliary functions.
\begin{notation}\label{PHnot} For all $n, k \in \rr$ such that $1 \leq k \leq n$, define
\beqs 
PH(n,k) := \ln\binom{2n}{n+k} - \ln\binom{2n}{n} + \dfrac{k^2}{n},  
\eeqs
where for $n_1\geq n_2$, 
$$
\binom{n_1}{n_2} = \dfrac{\Gamma(n_1+1)}{\Gamma(n_1-n_2+1)\Gamma(n_2+1)},
$$ 
and $\Gamma(x)$ is the Gamma function, defined for $x>0$ by 
$$
\Gamma(x) = \int_0^{\infty} t^{x-1}e^{-t}dt.
$$ 
It satisfies the well-known recurrence $\Gamma(x+1)\equiv x \Gamma(x).$
\end{notation} 
It is clear that $PH(n,k) \leq 0$ if and only if (\ref{simple1}) holds. 

\begin{notation} For all $n, k \in \rr$ such that $1 \leq k \leq n$, define
\begin{align}
DPH(n,k) &:= PH(n,k)-PH(n,k-1)\nonumber \\
&= \ln\left( \dfrac{n-k+1}{n+k} \right)  + \dfrac{2k-1}{n}.
\end{align}
\end{notation}

\begin{lemma}\label{DPHp} 
When $n \geq 19$, $DPH(n,k)$ is decreasing in $k$ when $k \geq \sqrt{n \ln{2}}$.
\end{lemma}
\begin{proof}
Clearly $DPH(n,k)$ is differentiable with respect to $k$ on the domain 
$n, k \in \rr$
such that $n>0$ and $0<k<n + 1/2$, 
with partial derivative given by
\beq \dfrac{\partial}{\partial k}DPH(n,k) = \dfrac{-2 k^2+2 k+n}{n
   \left(-k^2+k+n^2+n\right)}.\label{dk} \eeq
   It is easy to check that the denominator is positive on the given domain. 
Thus 
(\ref{dk}) is greater than 0 if and only if 
   $-2k^2 + 2k + n >0,$ which is equivalent to 
   $$ \dfrac{1}{2} \left(1-\sqrt{2n+1}\right)< k < \dfrac{1}{2} \left(1+\sqrt{2n+1}\right).$$
   
   Since we have that when $n\geq 19$, 
   $$\sqrt{n\ln{2}} > \dfrac{1}{2} \left(1+\sqrt{2n+1}\right),$$ $DPH(n,k)$ is decreasing in $k$ whenever $n\geq 19$.
\end{proof}

\begin{lemma}\label{DPHn} (a) For $0<\alpha<2/\sqrt{\ln 2}$ and all $n\geq 1$,
   \beq 
n-\alpha  \sqrt{n} \sqrt{\ln 2}+1 > 0. 
\label{qua} 
\eeq 
\fl
(b) For $\sqrt{3/(2\ln 2)} <\alpha < 2/\sqrt{\ln2}$ and $n$ large enough,
\beqs
\dfrac{d}{d n} DPH(n, \alpha\sqrt{n \ln{2}})>0.
\eeqs
\fl
(c) For $n \geq 3$, $DPH(n, \sqrt{3n})$ is increasing in $n$.

\fl
(d)  $DPH(n,\sqrt{3n})\to 0$ as $n\to\infty$.

\fl
(e) For all $n\geq 3$,
$
DPH(n, \sqrt{3n})<0.
$
\end{lemma}
\begin{proof} Part (a) holds because the left side of (\ref{qua}), as a quadratic
in $\sqrt{n}$, has the leading term $n=\sqrt{n}^2 >0$ and discriminant
$\Delta = \alpha^2 \ln 2 -4 < 0$ under the assumption. 

For part (b), by plugging  $k = \alpha\sqrt{n \ln{2}}$ into $DPH(n,k)$, 
we have
\beq\label{DPHplugin}  
DPH(n, \alpha\sqrt{n \ln{2}})=\dfrac{2 \alpha \sqrt{n \ln 2}-1}{n} + 
\ln \left(\dfrac{-\alpha \sqrt{n\ln 2}+n+1}{\alpha \sqrt{n \ln 2}+n}\right),
\eeq
which is well-defined by part (a).
It is differentiable with respect to $n$  with derivative given by
   
\begin{align} &\dfrac{d}{d n} DPH(n, \alpha\sqrt{n \ln{2}}) \nonumber \\
=&\dfrac{n\left(2 \alpha ^3 \ln^{\frac{3}{2}}(2)-3 \alpha \sqrt{\ln 2}\right)+\sqrt{n} \left(2-4\alpha ^2 \ln 2\right)+2 \alpha  \sqrt{\ln 2}}{2 n^2 \left(\alpha  \sqrt{\ln 2}+\sqrt{n}\right)\left(-\alpha  \sqrt{n}\sqrt{\ln 2}+n+1\right)} \label{dDPHn}. \end{align}


By part (a), the denominator 
$$
2 n^2 \left(\alpha  \sqrt{\ln2}+\sqrt{n}\right)
   \left(-\alpha  \sqrt{n}\sqrt{\ln 2}+n+1\right)
$$ 
is positive.
The numerator will be positive for $n$ large enough, since the coefficient of 
its leading term, 
$$
2 \alpha ^3 \ln^{{3}/{2}}(2)-3 \alpha \sqrt{\ln 2},
$$ 
is positive by the assumption 
$\alpha > \sqrt{{3}/{(2 \ln 2)}}$ in this part. So part (b) is proved. 


For part (c),
when $\alpha = \sqrt{3}/\sqrt{\ln{2}}$, we have 
$$
\dfrac{d}{d n} DPH(n, \sqrt{3n}) =
\dfrac{3 \sqrt{3} n-10 \sqrt{n}+2 \sqrt{3}}{2\left(\sqrt{n}+\sqrt{3}\right) 
\left(n-\sqrt{3} \sqrt{n}+1\right)
   n^2}.
$$
This is clearly positive when $3 \sqrt{3} n-10 \sqrt{n}+2\sqrt{3}\geq 0$, 
which always holds when $n\geq 3$. 
This proves part (c).

For part (d),
plugging 
$\alpha = \sqrt{{3}/{\ln{2}}}$ into (\ref{DPHplugin}), we have
\begin{eqnarray*}
&& \lim_{n \to \infty} DPH(n, \sqrt{3n}) \\
&=&\lim_{n \to \infty} \left( \dfrac{2 \sqrt{3n}-1}{n} + \ln \left(\dfrac{n-\sqrt{3n}+1}{n+ \sqrt{3n}}\right)
\right) \\
&=& 0,
\end{eqnarray*}
proving part (d). Part (e) then follows from parts (c) and (d).
\end{proof}

\begin{lemma}\label{L38}
For $n \geq 1$, $$DPH(n, \sqrt{n\ln{2}}) > 0.$$
\end{lemma}
\begin{proof}
By (\ref{dDPHn}) for $\alpha< 2/\sqrt{\ln 2}$, in this case $\alpha=1$,
 we have that
$$
\dfrac{d}{dn}DPH(n, \sqrt{n \ln{2}}) = \dfrac{n \left(2 \ln^{{3}/{2}}(2)
-3 \sqrt{\ln 2}\right)+\sqrt{n}(2-4 \ln 2)+2 \sqrt{\ln 2}}{2n^2
\left(\sqrt{n}+\sqrt{\ln 2}\right) \left(n-\sqrt{n} \sqrt{\ln 2}+1\right)}.
$$   
   The denominator is always positive for $n\geq 1$ by (\ref{qua}).
   The numerator as a quadratic in $\sqrt{n}$ has leading coefficient 
$2 \ln^{{3}/{2}}(2)-3 \sqrt{\ln 2}<0$. 
This quadratic also has a negative discriminant, so the numerator
is always negative when $n\geq1$.
   
   Similarly, we have
   \begin{eqnarray*}
&& \lim_{n \to \infty} DPH(n, \sqrt{n\ln 2}) \\
&=&\lim_{n \to \infty} \left( \dfrac{2 \sqrt{n \ln 2}-1}{n} + \ln \left(\dfrac{n-\sqrt{n \ln 2}+1}{n+ \sqrt{n \ln 2}}\right)
\right) \\
&=& 0.
\end{eqnarray*}
 Therefore $DPH(n, \sqrt{n\ln{2}})>0$ for all $n\geq 1$.
\end{proof}

Summarizing Lemmas \ref{DPHp}, \ref{DPHn}, and \ref{L38}, we have the 
following corollary:
\begin{corollary}\label{DPH}
For any fixed $n \geq 19$, $DPH(n,k)$ is decreasing in $k$ when $k\geq \sqrt{n \ln{2}}$. Furthermore, $$DPH\left(n,  \sqrt{n \ln 2} \right) > 0, \text{  } DPH\left(n, \sqrt{3n }\right)<0.$$
\end{corollary}


\begin{proposition} \label{thm1}
The inequality (\ref{mt}) holds for all integers $n,k$ such that $n \geq 108$ and $\sqrt{3n} \leq k \leq n$.
\end{proposition}

\begin{proof}
By Lemma \ref{up1}, the probability $P_{n,n,M}$ is bounded above by 
${2\binom{2n}{n+k}}/{\binom{2n}{n}}$. We here prove this proposition by 
showing that (\ref{simple1}) holds
for all integers $n$, $k$ such that $\sqrt{3n} \leq k \leq n$ and $n \geq 108$.

To prove (\ref{simple1}) is equivalent to proving
\beq
\ln{\binom{2n}{ n+k}} - \ln{\binom{2n}{ n}} + \dfrac{k^2}{n} < 0 \label{easier}
\eeq
for $k = t \sqrt{n}$ where $t \geq \sqrt{3}$, by Notation \ref{PHnot}.

Rewriting (\ref{easier}), we need to show that for $k\geq \sqrt{3n}$,
\beq 
\ln{\left(\dfrac{n!n!}{(n+k)!(n-k)!}  \right) + \dfrac{k^2}{n} < 0  } \label{easier2}.\eeq

We will use Stirling's formula with error bounds.
Recall that one form of such bounds \cite{ErrorTerm} states 
that 
$$
{\sqrt{2\pi}}
\exp\left({\dfrac{1}{12s}-\dfrac{1}{360s^3}-s}\right)s^{s+1/2}
 \leq s! \leq{\sqrt{2\pi}}\exp\left({\dfrac{1}{12s}-s}\right)
s^{s+1/2}
$$ 
for any positive integer $s$.
We plug the bounds for $s!$ into $\dfrac{n!n!}{(n+k)!(n-k)!}$, getting
$$
\dfrac{n!n!}{(n+k)!(n-k)!} \leq 
\dfrac{n^{2n+1}(n+k)^{-n-k-\frac{1}{2}}(n-k)^{k-n-\frac{1}{2}}
\exp\left({\dfrac{1}{6n}}\right)}
{\exp\left({\dfrac{1}{12}\left[\dfrac 1{n+k}+\dfrac{1}{n-k}\right]
-\dfrac{1}{360}\left[\dfrac 1{(n+k)^3} +\dfrac{1}{(n-k)^3}\right]}\right)}.
$$
By taking logarithms of both sides of the preceding inequality, we have
\begin{align}
\text{LHS of }(\ref{easier2}) \leq\ &  \dfrac{k^2}{n}+\dfrac{1}{6n} - 
\dfrac{1}{12}\left(\dfrac 1{n+k} + \dfrac 1{n-k}\right)
+\dfrac{1}{360}\left(\dfrac 1{(n+k)^3} +\dfrac{1}{(n-k)^3}\right) \nonumber \\
&- \left(n+k+\dfrac{1}{2}\right)\ln\left(1+\dfrac{k}{n}\right) - \left(n-k+\dfrac{1}{2}\right) \ln\left(1-\dfrac{k}{n}\right). \label{part1}
\end{align}

%

Plugging $k = t\sqrt{n}$ into the RHS of (\ref{part1}), we can write the result as $I_1+I_2+I_3$, where
\begin{align*}
I_1 =& -n \left(\left(1-\dfrac{t}{\sqrt{n}}\right)
   \ln
   \left(1-\dfrac{t}{\sqrt{n}}\right)+\left(\dfrac{t}{\sqrt{n}}+1\right) \ln
   \left(\dfrac{t}{\sqrt{n}}+1\right)\right), \\
I_2=&-\dfrac{1}{2} \left(\ln \left(1-\dfrac{t}{\sqrt{n}}\right)+\ln \left(\dfrac{t}{\sqrt{n}}+1\right)\right),\\
I_3=&-\dfrac{1}{12 \left(n-\sqrt{n}
   t\right)}-\dfrac{1}{12 \left(\sqrt{n}
   t+n\right)}+\dfrac{1}{360 \left(n-\sqrt{n}
   t\right)^3}+\dfrac{1}{360 \left(\sqrt{n}
   t+n\right)^3}\\
&+\dfrac{1}{6 n}+t^2.
\end{align*}
Then we want to prove that for $n$ large enough, \beq I_1 + I_2 + I_3 < 0.\label{wanttoprove} \eeq
Then as a consequence,  (\ref{easier2}) will hold.

By Corollary \ref{DPH} and the fact that $PH(n,k)$  is decreasing in $k$ for $n,k$ integers and $k \geq t \sqrt{n}$ where $t \geq \sqrt{3}$, if we can show that (\ref{wanttoprove}) holds for the smallest integer $k$ such that  $\sqrt{3n} \leq k \leq n$, then (\ref{easier}) will hold for all integers $\sqrt{3n} \leq k \leq n$.
Notice that if $k$ is the smallest integer not smaller than $\sqrt{3n}$, then 
$\sqrt{3n}\leq k < \sqrt{3n}+1.$ It is equivalent to say that
$\sqrt{3} \leq t \leq \left({\sqrt{3n}+1}\right)/{\sqrt{n}},$ and the RHS is smaller than 2 for 
all $n \geq 14$. So our goal now is to prove  (\ref{wanttoprove}) holds for 
all $n \geq 108$, as assumed in the proposition, and $\sqrt{3} \leq t < 2$.

By Taylor's expansion of $(1+x)\ln(1+x)+(1-x)\ln(1-x)$ around $x=0$, we find an upper bound for $I_1$, given by
\begin{eqnarray}
I_1 &=& -n\left( \sum_{i=1}^{\infty}\dfrac{t^{2i}}{n^{i}i(2i-1)} \right)\label{I1star}\\
    & <& -t^2-\dfrac{t^4}{6 n}-\dfrac{t^6}{15n^2}-\dfrac{t^8}{28 n^3}\nonumber.
\end{eqnarray}

For $I_2$, by using Taylor's expansion again, we have
\begin{eqnarray}
 I_2 = -\dfrac{1}{2}\left( \ln\left( 1-\dfrac{t^2}{n}\right)\right)
& =& \sum_{j=1}^{\infty} \dfrac 1{2j}\left( \dfrac{t^2}{n}\right)^j\label{Itwo} \\
 & \leq&  \dfrac{t^2}{2n} + \dfrac{t^4}{4n^2}+\dfrac{1}{2}R_3, \nonumber
\end{eqnarray}
where $R_3 = \sum_{j=3}^{\infty}{\dfrac 1 j}{\left( \dfrac{t^2}{n}\right)^j} 
< \dfrac{1}{3}\sum_{j=3}^{\infty}\left( \dfrac{t^2}{n}\right)^j = 
t^6/\left[{3}n^3\left( 1-\dfrac{t^2}{n} \right)\right].
 $

We only need to show (\ref{wanttoprove}) holds for all $\sqrt{3}\leq t <2$, 
and thus want to bound 
$t^6/\left[{3}n^3\left( 1-\dfrac{t^2}{n} \right)\right]$
by a sharp upper bound. This means we want $\dfrac{t}{\sqrt{n}}$ to be small. 
We have $n\geq 64$, which implies $\dfrac{t}{\sqrt{n}} < \dfrac{1}{4}$. Then 
we have an upper bound for $R_3$:
\beqs R_3 
\leq \dfrac{1}{3}\dfrac{t^6}{\left(15n^3/{16}\right)}.
\eeqs
It follows that
\beq I_2\leq \dfrac{t^2}{2n}+\dfrac{t^4}{4n^2}+\dfrac{8t^6}{45n^3}.\label{I_2star} \eeq

We now bound $I_3$ by studying two summands separately.
For the first part of $I_3$, we have
\begin{eqnarray*}
-\dfrac{1}{12 \left(n-\sqrt{n}
   t\right)}-\dfrac{1}{12 \left(\sqrt{n}
   t+n\right)} &=& -\dfrac{1}{12n}\left( \dfrac{1}{1- t/\sqrt{n}}+ \dfrac{1}{1+ t/\sqrt{n}}   \right) \\
   &=& -\dfrac{1}{6n}\left(1+ \left(\dfrac{t}{\sqrt{n}}\right)^2+  \left(\dfrac{t}{\sqrt{n}}\right)^4 + \dots \right)\\
   &<&  -\dfrac{1}{6n} - \dfrac{t^2}{6n^2}.
\end{eqnarray*}
For the second part of $I_3$, we have that when ${t}/{\sqrt{n}}\leq {1}/{4}$, 
\begin{eqnarray*}
\dfrac{1}{\left(\sqrt{n}t+n\right)^3}+\dfrac{1}{\left(n-\sqrt{n}t\right)^3} &=& \dfrac{1}{n^3}   \left( \dfrac{1}{\left(1+  t/\sqrt{n} \right)^3}+    \dfrac{1}{\left(1- t/\sqrt{n} \right)^3}  \right)\\
&<& \dfrac{1}{n^3}\left( \dfrac{1}{\left({5}/{4}\right)^3} +  
\dfrac{1}{\left({3}/{4}\right)^3}         \right) \\
&<& {3}/{n^3}.
\end{eqnarray*}
Therefore we have
  $$I_3 < - \dfrac{t^2}{6n^2}+\dfrac{3}{n^3}+ t^2.$$

Summing $I_1$ through $I_3$, we have
\begin{align}
I_1 + I_2 + I_3 <&t^2 -\dfrac{t^8}{28 n^3}-\dfrac{t^6}{15n^2}-\dfrac{t^4}{6 n}-t^2 +  \dfrac{t^2}{2 n} + \dfrac{t^4}{4 n^2}+ \dfrac{8t^6}{45n^3}  - \dfrac{t^2}{6n^2}+\dfrac{3}{n^3} \nonumber \\
<& \dfrac{1}{n} \left(\dfrac{t^2}{2}-\dfrac{t^4}{6}\right)+\dfrac 1{n^2}\left(-\dfrac{t^2}{6}+\dfrac{t^4}{4}-\dfrac{t^6}{15}\right)+ \dfrac 1{n^3}\left(3-\dfrac{t^8}{28}+\dfrac{8t^6}{45}\right)
\label{LHS}
\end{align} 
when $\dfrac{t}{\sqrt{n}} < \dfrac{1}{4}$, i.e., $n \geq 16 t^2$.

We now want to show that 
$I_1+I_2+I_3<0$ for all $n\geq 108$ and 
$\sqrt{3}\leq t < 2$.
We will consider the coefficients of $\frac{1}{n}, \frac{1}{n^2}, 
\frac{1}{n^3}$ in (\ref{LHS}).
The coefficient of $\frac{1}{n}$ is $\frac{t^2}{2}-\frac{t^4}{6}$, which is 
decreasing in $t$ when $\sqrt{3} \leq t < 2$; thus by plugging in 
$t=\sqrt{3}$, we have 
$$
\dfrac{t^2}{2}-\dfrac{t^4}{6} \leq 0.
$$
The coefficient of $\frac{1}{n^2}$ is 
$-\frac{t^6}{15}+\frac{t^4}{4}-\frac{t^2}{6}$, which is also decreasing in 
$t$ when $\sqrt{3} \leq t < 2$. Thus by plugging in $t=\sqrt{3}$, we have 
$$
-\dfrac{t^6}{15}+\dfrac{t^4}{4}-\dfrac{t^2}{6} \leq -\dfrac{1}{20}.
$$

The coefficient of $\frac{1}{n^3}$ is $-\frac{t^8}{28}+\frac{8t^6}{45}+ 3$. By calculation, we have that when $\sqrt{3} \leq t < 2$,
$$-\dfrac{t^8}{28}+\dfrac{8t^6}{45}+ 3 <5.4.$$ 

Thus when $n \geq 108 > 64$ and $\sqrt{3} \leq t < 2$, we have 
\beq I_1 + I_2+I_3 <  \dfrac{5.4}{n^3}-\dfrac{1}{20 n^2}.
\eeq

Therefore if we can show that for some $n$,
\beq 
\dfrac{5.4}{n^3}-\dfrac{1}{20 n^2} \leq 0,\label{numeq1}
\eeq 
then $I_1 + I_2 + I_3 < 0$ for those $n$. Solving (\ref{numeq1}), we obtain
$n \geq 108$.
\end{proof}

\begin{remark}
The coefficient of $\frac{1}{n}$ in (\ref{LHS}) is the same as the 
coefficient of $\frac{1}{n}$ in the Taylor expansion of $I_1+I_2+I_3$. 
So when the leading coefficient $\frac{t^2}{2}-\frac{t^4}{6}$ is positive, 
i.e., $t < \sqrt{3}$, the upper bound 
${2\binom{2n}{n+k}}/{\binom{2n}{n}}$ from Lemma \ref{up1} will tend to be 
larger than 
$e^{-{k^2}/{n}}$.
\end{remark}

\vspace{0.5cm}

Now we want to show that (\ref{mt}) holds for all integer pairs 
$(n, t \sqrt{n})$ with $\sqrt{\ln{2}}<t<\sqrt{3}$ and $n$ greater than some 
fixed value. By 
the argument in the remark, we need to choose another upper bound for 
$P_{n,n,M}$.

\begin{lemma}\label{up2}
We have $P_{n,n,M}\leq \dfrac{2\binom{2n}{n+k}-
\binom{2n}{n+2k}}{\binom{2n}{n}}$,
where $M=k/\sqrt{2n}$, $k=1,\dotsc,n$.
\end{lemma}

\begin{proof}
Let $A$ be the event that $\sup\sqrt{n}(F_n-G_n) \geq M$
and $B$ the event that $\inf\sqrt{n}(F_n-G_n) \leq -M.$ We want
an upper bound for $\Pr(A\cup B) = \Pr(A) + \Pr(B) - \Pr(A\cap B).$
Let $S_j$ be the value after $j$ steps of a simple, symmetric random
walk on the integers starting at $0$. Then $$\Pr(S_{2n}=2m) = 
\dfrac 1{4^n}\binom{2n}{n+m}$$ for $m=-n,-n+1,\cdots,n-1,n$.
By a well-known reflection principle 
we have nice exact expressions for $\Pr(A)$ and $\Pr(B)$, 
\beqs
\Pr(A) = \Pr(B) = \dfrac{\Pr(S_{2n}=2k)}{\Pr(S_{2n}=0)}= 
\dfrac{\binom{2n}{n+k}}{\binom{2n}{n}}.
\eeqs
Therefore we want
a lower bound for $\Pr(A\cap B)$. Let $C$ be the event that for
some $s<t$, $\sqrt{n}(F_n-G_n)(s)\geq M$ and $\sqrt{n}(F_n-G_n)(t) \leq -M$.
Then we can exactly evaluate $\Pr(C)$ by two reflections, e.g.\ \cite{GK},
specifically,
\beqs
\Pr(C) = \dfrac{\Pr(S_{2n}=4k)}{\Pr(S_{2n}=0)} = \frac{{{2n}\choose{n+2k}}}{\binom{2n}{n}},
\eeqs
 and $C\subset A\cap B$, so the bound holds.
\end{proof}

\begin{lemma} \label{twobinom}
Let $n,k$ 
be positive integers, $n\geq 372$, and 
$\sqrt{2n} < k = t\sqrt{n} \leq \sqrt{3n}$. Then 
$$
\binom{2n}{ n+2k}>\binom{2n}{n+k}e^{-3t^2 - 0.05}.
$$ 
\end{lemma}

\begin{proof}
By Stirling's formula with error bounds, we have
\beqs
\ln\left( \dfrac{\binom{2n}{n+2k}}{\binom{2n}{n+k}}  \right) 
=  
\ln \left( \dfrac{(n+k)!(n-k)!}{(n+2k)!(n-2k)!}  \right) > \ln(A_n)
\eeqs
where $A_n$ is defined as
\beqs
\dfrac{(n-k)^{n-k+\frac{1}{2}}(k+n)^{k+n+\frac{1}{2}}
\exp\left({\frac{1}{12}\left[\frac 1{k+n} +\frac{1}{n-k}\right]
-\frac{1}{360}\left[\frac 1{(k+n)^3}
+\frac{1}{(n-k)^3}\right]}\right)}
{\exp\left({\frac{1}{12 (2k+n)}+\frac{1}{12(n-2 k)}}\right)
(n-2 k)^{-2k+n+1/2} (2k+n)^{2k+n+1/2}},
\eeqs
and so
\begin{eqnarray}
\ln(A_n)&=& -\dfrac{1}{12(2 k+n)}-\dfrac{1}{12 (n-2 k)}+\dfrac{1}{12 (n-k)}+\dfrac{1}{12 (k+n)}\nonumber\\
&&   -\dfrac{1}{360 (n-k)^3}-\dfrac{1}{360 (k+n)^3}-\left(-2
   k+n+\dfrac{1}{2}\right) \ln (n-2 k)\nonumber\\
&&   +\left(-k+n+\dfrac{1}{2}\right)
   \ln (n-k)+\left(k+n+\dfrac{1}{2}\right) \ln (k+n)\nonumber\\
&& -\left(2
   k+n+\dfrac{1}{2}\right) \ln (2 k+n)\nonumber\\
   &=& I_4 + I_5, \label{4and5}
\end{eqnarray} where
\begin{eqnarray*}
I_4 &=& -\left(-2k+n+\dfrac{1}{2}\right) \ln (n-2 k)+\left(-k+n+\dfrac{1}{2}\right)\ln (n-k)\\
&&+\left(k+n+\dfrac{1}{2}\right) \ln (k+n)-\left(2k+n+\dfrac{1}{2}\right) \ln (2 k+n),\\
I_5&=&\dfrac{1}{12 (n-k)}+\dfrac{1}{12
   (k+n)}-\dfrac{1}{12 (2k+n)}-\dfrac{1}{12 (n-2 k)}\\
&&-\dfrac{1}{360(n-k)^3}-\dfrac{1}{360
   (k+n)^3}.
\end{eqnarray*}
   
Using again (\ref{I1star}) and (\ref{Itwo}), 
we have for $|x|<1$,

\begin{eqnarray*}
x^2+\dfrac{x^4}{6}& <& (1-x)\ln (1-x)+(x+1) \ln (x+1) \\
&<& x^2+\dfrac{x^4}{6}+\dfrac{1}{15}\sum_{i=3}^{\infty} x^{2i}
=x^2+\dfrac{x^4}{6}+\dfrac{x^6}{15(1-x^2)},
\end{eqnarray*} and also
\begin{eqnarray*}
-x^2
 &>&\ln (1-x)+\ln (x+1)\\
 &>& -x^2-\dfrac{1}{2}\sum_{i=2}^{\infty}x^{2i} = -x^2 - \dfrac{1}{2}\dfrac{x^4}{(1-x^2)}.
\end{eqnarray*}

So by plugging in $k = t \sqrt{n}$, we have that for $\dfrac{t}{\sqrt{n}}<\dfrac{1}{4}$,
\begin{align*} \allowdisplaybreaks
I_4 =\ & n\left(\left(1-\dfrac{k}{n}\right) \ln\left(1-\dfrac{k}{n}\right)+\left(\dfrac{k}{n}+1\right) \ln\left(\dfrac{k}{n}+1\right)\right)\\
&+ \dfrac{1}{2} \left(\ln \left(1-\dfrac{k}{n}\right)+\ln\left(\dfrac{k}{n}+1\right)\right)\\
   &- n\left( \left(1-\dfrac{2 k}{n}\right) \ln \left(1-\dfrac{2k}{n}\right)+\left(\dfrac{2 k}{n}+1\right) \ln \left(\dfrac{2k}{n}+1\right)\right)\\
   &- \dfrac{1}{2} \left(\ln \left(1-\dfrac{2 k}{n}\right)+\ln \left(\dfrac{2k}{n}+1\right)\right)\\
>\ &n \left( \left(\dfrac{t}{\sqrt{n}}\right)^2+\dfrac 16\left(\dfrac{t}{\sqrt{n}}\right)^4 \right)  
-\dfrac{1}{2}\left(\left(\dfrac{t}{\sqrt{n}}\right)^2 +    \dfrac{8}{15}\left(\dfrac{t}{\sqrt{n}}\right)^4  \right)\\
&- n \left(  \left(\dfrac{2t}{\sqrt{n}}\right)^2+ \dfrac 16 \left(\dfrac{2t}{\sqrt{n}}\right)^4+\dfrac 4{45}\left(\dfrac{2t}{\sqrt{n}}\right)^6 \right) +\dfrac{1}{2}\left(\dfrac{2t}{\sqrt{n}}\right)^2 \\ 
=\ & t^2+ \dfrac{t^4}{6n}-\dfrac{t^2}{2n}-\dfrac{4t^4}{15n^2}-4t^2-
\dfrac{8t^4}{3n}-\dfrac{256t^6}{45n^2}+\dfrac{2t^2}{n}\\
=\ &  -\dfrac 1{n^2}\left(\dfrac{256 t^6}{45}+\dfrac{4 t^4}{15}\right)
   +\dfrac 1n\left({\dfrac{3 t^2}{2}-\dfrac{5 t^4}{2}}\right)-3 t^2.
\end{align*}

Now we proceed to find a lower bound for $I_5$.
 For all $k\leq n/8$, in other words $t:= k/\sqrt{n}$ such that
$8t\leq\sqrt{n}$,
\begin{align*}\allowdisplaybreaks
I_5=\ & \dfrac{1}{12}\left(\dfrac 1{n-k}+\dfrac 1
   {k+n}-\dfrac{1}{(2k+n)}-\dfrac{1}{(n-2 k)}\right)\\
    &-\dfrac{1}{360}\left(\dfrac 1{(k+n)^3} + \dfrac 1{(n-k)^3}\right) 
\\
   =\ & \dfrac{1}{12}\left(\dfrac 1{\sqrt{n}t+n}+\dfrac{1}{n-\sqrt{n}t} 
   - \dfrac{1}{2 \sqrt{n} t+n}-\dfrac 1{n-2 \sqrt{n} t}\right)\\
   &- \dfrac{1}{360}\left(\dfrac 1{(\sqrt{n} t+n)^3} +\dfrac{1}{\left(n-\sqrt{n} t\right)^3}\right)\\
   =\ & \dfrac{1}{6 (n-t^2)} - \dfrac{1}{6 \left(n-4 t^2\right)} - \dfrac{n+3 t^2}{180 n \left(n-t^2\right)^3 } \\
   >\ & \dfrac{1}{6n}-\dfrac{1}{3n}-\dfrac{n+3t^2}{90n^4} \\   
   =\ & -\dfrac{1}{6n}-\dfrac{1}{90n^3}-\dfrac{t^2}{30n^4}.
\end{align*}
Since $t \leq \sqrt{3}$, we know that as long as $n \geq 192,$ the condition $8t \leq \sqrt{n}$ will hold.

Adding our lower bounds for $I_4$ and $I_5$, we have that when $n \geq 192$ and $\sqrt{\ln 2} \leq t \leq \sqrt{3}$,
\begin{eqnarray} 
I_4+I_5&>&-\dfrac{t^2}{30 n^4}-\dfrac{1}{90 n^3}-\frac{1}{n^2}\left(\dfrac{256 t^6}{45}+\dfrac{4}{15}
t^4\right)-\frac{1}{n}\left(\dfrac{5 t^4}{2}-\dfrac{3 t^2}{2}+\dfrac{1}{6}\right)-3 t^2 \nonumber \\
   &>& -3t^2 - \gamma,\label{onestar}
\end{eqnarray} 
for some $\gamma$. When $\gamma=0.05$, we want to show that for 
$n$ large enough, (\ref{onestar}) always holds. In other words, 
we need 
\beq 
0.05 > \dfrac{t^2}{30 n^4}+\dfrac{1}{90 n^3}+
\dfrac{1}{n^2}\left(\dfrac{256 t^6}{45}+\dfrac{4}{15}
 t^4 \right)+\dfrac{1}{n}\left({\dfrac{5 t^4}{2}-\dfrac{3 t^2}{2}+
\dfrac{1}{6}}\right).
\label{gammarange}
\eeq
Notice that when $\sqrt{\ln{2}}<t<\sqrt{3}$, the coefficient $\frac{5 t^4}{2}-\frac{3
   t^2}{2}+\frac{1}{6}$ is positive and is increasing in $t$; the RHS of (\ref{gammarange}) is increasing in $t$ and decreasing in $n$. Thus we just need to make sure the inequality holds for $t=\sqrt{3}$. Therefore we need 
\beq 0.05 > \dfrac{1}{10 n^4}+\dfrac{1}{90 n^3}+\dfrac{156}{n^2}+\dfrac{109}{6 n}.\label{34}\eeq Solving (\ref{34}) numerically, we find that it holds for 
$n \geq 372.$

Therefore, by (\ref{4and5}) and (\ref{onestar}), we have shown that when $n \geq 372$, 
$$
\ln\left[{\binom{2n}{ n+2k}}/{\binom{2n}{n+k}}\right]>-3t^2 - 0.05,
$$ 
for $k = t\sqrt{n}$ and $\sqrt{\ln{2}} < t < \sqrt{3}$, proving 
Lemma \ref{twobinom}.
\end{proof}


\begin{proposition} \label{thm2} Let $k = t\sqrt{n}$, where $\sqrt{\ln{2}} < t < \sqrt{3}$, and $k,n$ integers. Then the inequality
$$
\left[{2\binom{2n}{n+k}-\binom{2n}{n+2k}}\right]/{\binom{2n}{n}} < 
2\exp\left({-{k^2}/{n}}\right)
$$ holds for $n \geq 6395$.
\end{proposition}

\begin{proof}
By Lemma \ref{twobinom},
it will suffice to show that for $n\geq 6395 > 372$,
\beq{\binom{2n}{n+k}\left(1-e^{-3t^2 - 0.05}/2\right)}/{\binom{2n}{n}} < 
\exp(-{k^2}/{n}).\label{threestar}\eeq

Rewriting (\ref{threestar}) by taking logarithms of both sides, we just need to show
\beqs 
\ln{\binom{2n}{n+k}}-\ln{\binom{2n}{n}}+\dfrac{k^2}{n} + \ln{\left(1-e^{-3t^2 - 0.05}/2 \right)}<0.
\eeqs 

By (\ref{easier2}), (\ref{part1}), and (\ref{LHS}), we have that 
$$\ln{\binom{2n}{n+k}}-\ln{\binom{2n}{n}}+\dfrac{k^2}{n} < \dfrac{3-\frac{t^8}{28}+\frac{4t^6}{45}}{n^3}+\frac{-\frac{t^2}{6}+\frac{t^4}{4}-\frac{t^6}{15}}{n^2}+\dfrac{\frac{t^2}{2}-\frac{t^4}{6}}{n} $$ for $n > 16t^2$.
So now we just need 
\beq \dfrac{3-\frac{t^8}{28}+\frac{4t^6}{45}}{n^3}+\dfrac{-\frac{t^2}{6}+\frac{t^4}{4}-\frac{t^6}{15}}{n^2}+\dfrac{\frac{t^2}{2}-\frac{t^4}{6}}{n}+ \ln{\left(1-e^{-3t^2 - 0.05}/2\right)}<0. \label{now} \eeq

When $\sqrt{\ln{2}}<t<\sqrt{3}$, the coefficient
$\dfrac{t^2}{2}-\dfrac{t^4}{6}>0$. Next, using $t<\sqrt{3}$,
\begin{eqnarray*} &&\dfrac 1{n^3}\left(3-\dfrac{t^8}{28}+\dfrac{4t^6}{45}\right)
+\dfrac 1{n^2}\left(-\dfrac{t^2}{6}+\dfrac{t^4}{4}-\dfrac{t^6}{15}\right)
+\dfrac 1{n}\left(\dfrac{t^2}{2}-\dfrac{t^4}{6}\right)
\\
&<& \dfrac 1{n} \left(\dfrac{t^2}{2}-\dfrac{t^4}{6}\right)
+ \dfrac{t^4}{4n^2}+
\dfrac 1{n^3}\left(3+\dfrac{4t^6}{45}\right)\\
&<&  \dfrac 1{n}\left(\dfrac{t^2}{2}-\dfrac{t^4}{6}\right)
 + \dfrac{9}{4n^2}+\dfrac{27}{5n^3}.
\end{eqnarray*}

Clearly, the maximum value of $\ln{\left(1-e^{-3t^2 - 0.05}/2\right)}$ for  $\sqrt{\ln{2}}\leq t \leq \sqrt{3}$ is achieved when $t = \sqrt{3}$. Plugging in $t = \sqrt{3}$ into $\ln{\left(1-e^{-3t^2 - 0.05}/2\right)}$, we have
$$\ln{\left(1-e^{-3t^2 - 0.05}/2\right)} \leq -0.0000586972.$$

Now we find the maximum value of $\frac{t^2}{2}-\frac{t^4}{6}$ for $\sqrt{\ln{2}}\leq t \leq \sqrt{3}$. The derivative with respect to $t$ is 
$t-\dfrac{2 t^3}{3},$ which equals zero when $t = \sqrt{1.5}$. This critical point corresponds to the maximum value of $\frac{t^2}{2}-\frac{t^4}{6}$ for $\sqrt{\ln{2}}<t<\sqrt{3}$, and this maximum value is $0.375$. 

Accordingly, when $\sqrt{\ln{2}}<t<\sqrt{3}$,  
$$
\text{LHS of (\ref{now}) } < -0.0000586972+\dfrac{9}{4n^2}+
\dfrac{39}{5n^3}+\dfrac{3}{8n}.
$$
We just need 
\beq\label{crazyy}
-0.0000586972+ \dfrac{9}{4n^2}+\dfrac{39}{5n^3}+\dfrac{3}{8n}<0.
\eeq 
The LHS of (\ref{crazyy}) is decreasing in $n>0$. By numerically solving the 
inequality in $n$ we have that $n\geq 6395.$ 
Therefore we have proved that when $n > 6395,$ the original inequality (\ref{mt}) holds for all positive integer pairs $(k,n)$ such that $\sqrt{n \ln 2}<k<\sqrt{3n}$ and $k\leq n$.
\end{proof}

Recall that by (\ref{ksmall}), the inequality (\ref{mt}) holds for all $k \leq \sqrt{n\ln{2}}$. Combining 
Propositions \ref{thm1} and \ref{thm2}, we have the following conclusion.
\begin{theorem} 
$(a)$
 When $n \geq 6395,$ $(\ref{mt})$ holds for all $(n,k)$ such that $0 \leq k \leq n$.
\par\noindent
$(b)$ When $6395>n \geq 372$, $(\ref{mt})$ holds for all integer pairs $(n,k)$ such 
that $0 \leq k \leq \sqrt{n\ln{2}}$ and $\sqrt{3n} < k \leq n$.
\end{theorem}
 Then by computer searching for the rest of the integer pairs $(n,k)$, 
namely, $1 \leq k \leq n$ when $1\leq n\leq 371$
and $\sqrt{n\ln 2}<k\leq \sqrt{3n}$ when $372\leq n<6395$,
 we are able to find the finitely many  counterexamples to the inequality
(\ref{mt}), and thus prove Theorem \ref{meqn}.

 \section{Treatment of $m\neq n$}\label{mneqnsec}
\subsection{One- and two-sided probabilities}
For given positive integers $1\leq m\leq n$ and  $d$ 
with $0<d\leq 1$,
let $pv_{os}$ be the one-sided probability 
\beq\label{oneside}
pv_{os}(m,n,d) = \Pr(\sup_x(F_m-G_n)(x)\geq d) = \Pr(\inf_x(F_m-G_n)(x)\leq -d),
\eeq
where the equality holds by symmetry (reversing the order of the observations
in the combined sample). 
Let the two-sided probability ($p$-value) be
$$
P(m,n;d) := \Pr(\sup_x|(F_m-G_n)(x)|\geq d).
$$
The following is well known, e.g.\ for part (b),
\cite[p.\ 472]{Hodges}, and easy to check:
\begin{theorem}\label{oneandtwoside}
For any positive integers $m$ and $n$ and any $d$ with $0<d\leq 1$
we have
\fl
$(a)$ $
pv_{os}(m,n,d) \leq P(m,n;d) \leq pv_{ub}(m,n,d):= 2pv_{os}(m,n,d).
$
\fl
$(b)$  If $d>1/2$, $P(m,n;d) = pv_{ub}(m,n,d)$.
\end{theorem}
\subsection{Computational methods}\label{computation}
  To compute $p$-values $P(m,n;d)$ for the 2-sample test
for $d\leq 1/2$ we used the Hodges (1957) ``inside''
algorithm, for which Kim and Jennrich \cite{Kimjennrich}
gave a Fortran program and tables computed with it
for $m\leq n \leq 100$.
We further adapted the program to double precision.
The method seems to work reasonably well for $m\leq n\leq 100$;
 for $n=2m$ with $m\leq 94$ and $d=(m+1)/n$ it still gives one or 
two correct significant digits, see Table \ref{tab:pvo}.
The inside method finds $p$-values  
$\Pr(D_{m,n}\geq d)$ as $1-\Pr(D_{m,n}<d)$. When $p$-values are very small,
e.g.\ of order $10^{-15}$,
the subtraction can lead to substantial or even total loss of significant 
digits,
due to subtracting numbers very close to 1 from 1 (again see Table
\ref{tab:pvo}).

The one-sided probabilities $pv_{os}(m,n,d)$ and thus $P(m,n;d)$ for
$d>1/2$ by Theorem \ref{oneandtwoside}(b) can be computed 
by an analogous ``outside'' method with only additions and multiplications
(no subtractions), so it can compute much smaller probabilities very accurately.
The smallest probability needed for computing the results of the paper is
$\Pr(D_{300,600}\geq 1)$ which was evaluated by the outside program
as $1.147212371856\cdot 10^{-247}$, confirmed to the given number (13) of
significant digits by evaluating $2/{\binom {900}{300}}$. Moreover the ratio
of this to $2\exp(-2M^2)$ is about $3\cdot 10^{-74}$, so great accuracy
in the $p$-value is not needed to see that the ratio is small.
 For $m=n$ we can compare results of the outside method to those
found from the Gnedenko--Korolyuk formula in Proposition 
\ref{realprobability}. For $\Pr(D_{500,500}\geq 0.502)$ the outside
method needs to add a substantial number of terms. It gives
$1.87970906825\cdot 10^{-57}$ which agrees with the Gnedenko--Korolyuk
result to the given accuracy.

For large enough $m,n$ there will be an interval of values of
$d$, 
\beq\label{range}
d_0(m,n)\leq d\leq 1/2,
\eeq
in which the $p$-values are too small to compute accurately
by the inside method. We still have the possibility of
verifying the DKWM inequality in these ranges using
Theorem \ref{oneandtwoside}(a) if we can show that
\beq\label{useupbds}
pv_{ub}(m,n,d) \leq 2\exp(-2M^2)
\eeq
where as usual $M=\sqrt{mn/(m+n)}d$, and did so
computationally for $100\leq m<n\leq 200$ and $190 \leq n=2m\leq 600$
as shown by ratios less than 1 in the last columns of
Tables \ref{tab:HTH} and \ref{tab:D2} respectively.

With either  the inside or outside method, evaluation of an 
individual probability takes $O(mn)$ computational
steps, which is more (slower) than for $m=n$. For $mn$ large,
rounding errors accumulate, which especially affect the inside
method.
Moreover, to find the $p$-values for all possible values of $D_{mn}$,
in the general case that $m$ and $n$ are relatively prime, as in
a study like the present one, gives another factor of $mn$ and so
takes $O(m^2n^2)$ computational steps.

The algorithm does not require storage of $m\times n$ matrices.
Four vectors of length $n$, and various individual variables,
are stored at any one time in the computation.

For $n=2m$, the smallest possible $d>1/2$ is $d=(m+1)/n$.
Let $pvi$ and $pvo$ be the $p$-value $\Pr(D_{m,n}\geq d)$ as computed
by the inside and outside methods respectively. Let the relative
error of $pvi$ as an approximation to the more accurate $pvo$ be
$reler = \left|\dfrac{pvi}{pvo} - 1\right|$. For $n=2m$, $m=1,\dotsc,120$,
and $d = (m+1)/n$, the following $m=m_{\max}$ give larger $reler$ than
for any $m<m_{\max}$, with the given $pvo$.

\newpage

\begin{center}
{\small
\tablecaption{$p$-values for $n=2m$, $d=(m+1)/n$}
\tablehead{\bfseries  $m_{\max}$      &  $reler$     & $pvo$ \\ \hline}
\tabletail{\hline \multicolumn{3}{r}{\emph{Continued on next page}}\\}
\tablelasttail{\hline}\label{tab:pvo}
\begin{supertabular}{rll}
    10    &  $5.55\cdot 10^{-15}$ &  $0.0290$  \\
    20    &  $7.88\cdot 10^{-13}$ &  $8.94\cdot 10^{-4}$  \\
    28    &  $2.04\cdot 10^{-12}$ &  $5.48\cdot 10^{-5}$  \\
    40    &  $1.32\cdot 10^{-9}$ &  $8.29\cdot 10^{-7}$  \\
    49    &  $6.51\cdot 10^{-9}$ &  $3.58\cdot 10^{-8}$  \\
    60    &  $1.01\cdot 10^{-6}$ &  $7.66\cdot 10^{-10}$  \\
    70    &  $4.76\cdot 10^{-5}$ &  $2.32\cdot 10^{-11}$  \\
    80    &  $2.19\cdot 10^{-3}$ &  $7.07\cdot 10^{-13}$  \\
    93    &  $0.063$ &  $7.52\cdot 10^{-15}$  \\
    95    &  $0.109$ & $3.74\cdot 10^{-15}$ \\       
    98    &  $0.525$ &  $1.31\cdot 10^{-15}$  \\
    100   &  $1.045$ &  $6.52\cdot 10^{-16}$  \\
    105   &  $9.758$ &  $1.14\cdot 10^{-16}$  \\
    120   &  $2032.4$ &  $6.01\cdot 10^{-19}$  \\
    \end{supertabular}}
\end{center}

\

The small relative errors for $m\leq 10,20$, or 40, indicate
that the inside and outside programs algebraically 
confirm one another. As $m$ increases, $pvo$ becomes
smaller and $reler$ tends to increase until for $m=100$,
$pvi$ has no accurate significant digits. For $m=105$,
$pvi$ is off by an order of magnitude and for $m=120$ by
three orders. For $m=122$, $n=244$, and $d=123/244$,
for which $pvo = 2.99\cdot 10^{-19}$, $pvi$ is negative,
 $-4.44\cdot 10^{-16}$. In other words, the inside computation
gave $\Pr(D_{122,244} < 123/244) \doteq 1 + 4.44\cdot 10^{-16}$
which is useless, despite being accurate to 15 decimal places.

%

Of course, $p$-values of order $10^{-15}$ are not needed for applications of the
Kol\-mo\-go\-rov--Smirnov test even to, say, tens of thousands of 
simultaneous hypotheses as in genetics, but in this paper we are concerned with 
the theoretical issue of validity of the DKWM bound.

\subsection{Details related to Facts \ref{mneqn}, 
\ref{doubles}, and \ref{lowbdreler}}\label{Facts2-3}
Fact \ref{mneqn}(b) states that for $1\leq m<n\leq 3$ the
DKWM inequality fails. The following lists $r_{\max}(m,n)>1$
for each of the three pairs and the $d_{\max}$, equal to 1 in these
cases, for which $r_{\max}$ is attained.
\vskip .2in

\begin{minipage}[c]{3.7cm}\small{
\begin{tabbing}
 mcol \ \=\ ncol \ \=\ 1.2645561 \ \=\ \ d \kill
 $m$ \>$n$ \ \>\ \ \  \  \ $r_{\max}$ \ \> $d_{\max}$ \\
 1 \ \>2 \>1.264556 \>\ \  1 \\
 1 \ \>3 \>1.120422 \>\ \ 1 \\
 2 \ \>3 \>1.102318 \>\ \ 1 \\
\end{tabbing}}
\end{minipage}

Fact \ref{mneqn}(a) states that if $1\leq m<n\leq 200$
and $n\geq 4$, the DKWM inequality holds.
Searching through the specified $n$ for each $m$, we got the following.

For $m=1,2,$ the results of Fact \ref{mneqn}(f) as stated were found.
%

For $3\leq m\leq 199$ and $m<n\leq 200$ we searched over $n$ for
each $m$, finding $r_{\max}(m,n)$ for each $n$ and the
$n=n_{\max}$ giving the largest $r_{\max}$. Tables \ref{tab:D1} 
and \ref{tab:HTH} in Appendix B
show that all $r_{\max}<1$, completing the evidence for
Fact \ref{mneqn}(a), and were always found at
$n_{\max} = 2m$ for $m\leq 100$, as Fact \ref{mneqn}(c) states.

For Fact \ref{mneqn} (d) and (e) and Fact
\ref{doubles}, the results stated can be seen in
Tables \ref{tab:HTH} and \ref{tab:D2}.

Fact \ref{doubles}(a) in regard to relative minima of $r_{\max}$
is seen to hold in Table \ref{tab:D1}. Increasing $r_{\max}$
for $16\leq m\leq 300$ is seen in Tables \ref{tab:D1} and \ref{tab:D2}.
Fact \ref{doubles}(b) is seen in Table \ref{tab:D2}.

In Fact \ref{doubles}(c), the minimal $r_{\max}(m,2m)$ for $m\geq 101$
is at $m=101$ by part (a) with value $0.973341$ in Table
\ref{tab:D2}. The largest $r_{\max}$ in Table \ref{tab:HTH} for
$m\geq 101$ is $0.949565<0.973341$ as seen with the aid of
Fact \ref{mneqn}(d). For Fact \ref{doubles}(d), one sees that
$k_{\max}$ is nondecreasing in $m$ in Tables \ref{tab:D1} and
\ref{tab:D2}.

Regarding Fact \ref{lowbdreler},
the relative error of the DKWM bound as an approximation of a
$p$-value, namely
\beq\label{relerdkwm}
reler(dkwm,m,n,d) := \frac{2\exp(-2M^2)}{P_{m,n,M}} - 1,
\eeq
where $M$ is as in (\ref{possM}) with $d= k/L_{m,n}$, is bounded below 
for any possible $d$ by
\beq\label{dkwmrelowbd}
reler(dkwm,m,n,d) \geq \frac 1{r_{\max}(m,n)} - 1.
\eeq
 From our results, over the
given ranges, the relative error has the best chance to be small
when $n=m$ and the next-best chance when $n=2m$. On the other
hand, in Table \ref{tab:HTH} in Appendix B, where
$
rmaxx = rmaxx(m)  = \max_{m<n\leq 200}r_{\max}(m,n),
$
we have for each $m,n$ with $100 < m < n\leq 200$ and possible $d$ that
\beq\label{lowbdwrmaxx}
reler(dkwm,m,n,d) \geq \frac 1{rmaxx(m)} - 1.
\eeq
Thus Fact \ref{lowbdreler}(a) holds by Fact \ref{doubles}(c)
and the near-equality of $\beta(M)$ and $2\exp(-2M^2)$ if either
is $\leq 0.05$, as in the Remark after (\ref{brbrlim}).   
Fact \ref{lowbdreler}(b) holds similarly by inspection of
Table \ref{tab:HTH}.
\subsection{Conservative and approximate $p$-values}
Whenever the DKWM inequality holds, the DKWM bound $2\exp(-2M^2)$ 
provides simple, conservative $p$-values. 
 The  asymptotic $p$-value $\beta(M)$ given in (\ref{brbrlim})
is very close to the DKWM bound 
in case of significance level $\leq 0.05$ or less,
as noted in the Remark just after (\ref{brbrlim}).

In general, by Fact \ref{lowbdreler} for example,
 using the DKWM bound as an approximation
can give overly conservative $p$-values. We looked at
$m=20$, $n=500$. For $\alpha = 0.05$ the correct critical
value for $d=k/500$ is $k=151$ whereas the approximation
would give $k=155$; for $\alpha = 0.01$ the correct critical
value is $k=180$ but the approximation would give $k=186$.
For $180\leq k\leq 186$ the ratio of the true $p$-value to its
DKWM approximation decreases from $0.731$ down to $0.712$.
 
Stephens \cite{Stephens70} proposed that {\em in the one-sample
case}, letting $N_e:= n$ and
\beq\label{Fdef}
F := \sqrt{N_e} + 0.12 + 0.11/\sqrt{N_e},
\eeq
one can approximate $p$-values by $\Pr(D_n\geq d) \sim \beta(Fd)$
for $0<d\leq 1$, with $\beta$ from (\ref{brbrlim}). Stephens gave
evidence that the approximation works rather well. 
In the one-sample
case the distributions of the statistics $D_n$ and $K_n$ are continuous 
for fixed $n$ and vary rather smoothly with $n$.

Some other sources, e.g.\  \cite[pp.\ 617-619] {recipes},
propose in the two-sample case setting $N_e = mn/(m+n)$,
defining $F:= F_{m,n}$ by (\ref{Fdef}), and approximating
$\Pr(D_{m,n}\geq d)$ by $S_{\textrm{pli}}:= \beta(Fd)$
[``Stephens approximation plugged into'' two-sample].  
Since $F$ in (\ref{Fdef})
is always larger than
$\sqrt{N_e}$, $S_{\textrm{pli}}$ is always less than the
asymptotic probability $\beta(M)$ for $M=\sqrt{N_e}d$ which, in turn, is
always less than the DKWM approximation $2\exp(-2M^2)$.
The approximation $S_{\textrm{pli}}$ is said in at least two sources we have seen
(neither a journal article)
to be already quite good for $N_e\geq 4$. That may well be true in the
one-sample case. In the two-sample case it may be true when
$1<m\ll n$ but not when $n\sim m$. 
Table \ref{tab:approx} compares the two approximations 
$dkwm=2\exp(-2M^2)$  and $S_{\textrm{pli}}$ to critical $p$-values for some
pairs $(m,n)$.
For $m=n$, and to a lesser extent when $n=2m$, it seems that
$dkwm$ is preferable. For other pairs,
$S_{\textrm{pli}}$ is. For the six pairs $(m,n)$ with
$L_{m,n}=n$ or $2n$, $S_{\textrm{pli}} < pv.$
For the other two (relatively prime) pairs, $pv < S_{\textrm{pli}}$.
For $m=39,n=40$, $S_{\textrm{pli}}$
has rather large errors, but those of $dkwm$ are much larger.

In Table \ref{tab:approx}, $d=k/L_{m,n}$ 
and $pv$ is the 
correct $p$-value.
After each of the two approximations, $dkwm$ and $S_{\textrm{pli}}$, is 
 its relative error $reler$ as an approximation of $pv$.
\begin{center}
\small{
\tablecaption{Comparing two approximations to $p$-values}
\tablehead{\bfseries  $m$     &  $n$& $N_e$   & $k$  & $d$ & $pv$ & $dkwm$ & 
$reler$& $S_{\textrm{pli}}$ & $reler$ \\ \hline}
\tabletail{\hline \multicolumn{10}{r}{\emph{Continued on next page}}\\}
\tablelasttail{\hline} \label{tab:approx}
\begin{supertabular}{rrrrllllll}
40 & 40 &20 &12 &.3 &.05414 &.05465 &.0094 &.04313 &.2033  \\
40 & 40 &20 &13 &.325 &.02860 &.02925 &.0226 &.02216 &.2253 \\
40 & 40 &20 &14 &.35 &.014302 &.01489 &.0413 &.01079 &.2453 \\
40 & 40 &20 &15 &.375 &.006761 &.00721 &.0669 &.00498 &.2628 \\
200 & 200 &100 &27 &.135 &.05214 &.05224 &.0020 &.04745 &.0899 \\
200 & 200 &100 &28 &.14 &.03956 &.03968 &.0030 &.03578 &.0955 \\
200 & 200 &100 &32 &.16 &.011843 &.01195 &.0092 &.01044 &.1183 \\
200 & 200 &100 &33 &.165 &.008539 &.00864 &.0113 &.00748 &.1240 \\
25 & 50 &$16.67$ &16 &.32 &.06066 &.06586 &.0858 &.05129 &.1545 \\
25 & 50 &$16.67$ &17 &.34 &.03847 &.04242 &.1025 &.03198 &.1687 \\
25 & 50 &$16.67$ &19 &.38 &.014149 &.01624 &.1479 &.01141 &.1933 \\
25 & 50 &$16.67$ &20 &.4 &.008195 &.00966 &.1783 &.00653 &.2029 \\
39 & 40 &19.75 & 456 &.2923 & .05145 &.06847 &.3309 &.05476 &.0644  \\
39 & 40 &19.75 & 457 &.2929 &.04968 &.06746 &.3579 &.05390  &.0850  \\
39 & 40 &19.75 & 541 &.3468 &.010159  &.01731 &.7036 &.01264 &.2439  \\
39 & 40 &19.75 & 542 &.3474 &.009849 &.01701 &.7267 &.01240 &.2593  \\
20 & 500 &$19.23$ &150 &.3 &.05059 &.06276 &.2406 &.04973 & .0171 \\
20 & 500 &$19.23$ &151&.302 &.04817&.05992& .2439& .04733& .0175 \\
20 & 500 &$19.23$ &179&.358 &.010608&.01446& .3634& .01038& .0214\\ 
20 & 500 &$19.23$ &180 &.36 &.009998 &.01368 &.3688 &.009787 &.0211 \\
21 & 500 &$20.15$ &3074&.29276&.050052 &.06319 &.2626 &.050410 &.0072 \\
21 & 500 &$20.15$ &3076$^*$&.29295 &.049882 &.06291 &.2612 &.050170 &.0058 \\
21 & 500 &$20.15$ &3686&.35105 &.010040 &.01392 &.3869 &.010062 &.0022 \\
21 & 500 &$20.15$ &3687 &.35114 &.009979 &.01389 &.3917 &.010033 &.0054 \\
100 &500 &83.33 &73 &.146 &.0534470 &.0572963 &.07202 &.051661 &.03343   \\
100 &500 &83.33 &74 &.148 &.0483882 &.0519476 &.07356 &.0467046 &.03479 \\
100 &500 &83.33 &88 &.176 &.0104170 &.0114528 &.09943 &.0098532 &.05413 \\
100 &500 &83.33 &89 &.178 &.0092390 &.010178 &.1016 &.0087264 &.05548 \\
400 &600& 240 &104 &.08667 &.0521403 &.0543568 &.04251 &.051221 &.01763  \\
400 &600& 240 &105 &.0875 &.0486074 &.0506988 &.04303&.047719& .01827 \\
400 &600& 240 &125&.10417&.0103748& .0109416& .05463&.0100418& .03210 \\
400 &600& 240 &126 &.105 &.0095362 &.0100634 &.05528 &.0092231 &.03283  \\
 \end{supertabular} }%
\end{center}
\small{(* For $(m,n)=(21,500)$, the value $k=3075$ is not possible.)}
\newpage
The pair $(400,600)$ was included in Table \ref{tab:approx}
because, according to Fact 2(d),
the ratio $n/m = 3/2$ seemed to come next after $1/1$ and $2/1$ in producing
large $r_{\max}$, and so possibly small relative error for $dkwm$ as
an approximation to $pv$, and $r_{\max}$ was increasing in the range
computed for this
ratio, $m = 102, 104,...,132$. Still, the relative errors of $S_{\textrm{pli}}$
in Table \ref{tab:approx} are smaller than for $dkwm$.

It is a question for further research whether the usefulness of
$S_{\textrm{pli}}$, which we found for $m=20$ or 21 and $n=500$, 
extends more generally to cases where $m$ is only moderately large 
and $m \ll n$.

%
 
\subsection{Obstacles to asymptotic expansions}
This is to recall an argument of Hodges \cite{Hodges}.
Let
$$
Z^+ := Z_{m,n}^+ := \sqrt{\dfrac{mn}{m+n}}\sup_x(F_m-G_n)(x),
$$
a one-sided two-sample Smirnov statistic. There is the well-known
limit theorem that for any $z>0$, if $m,n\to\infty$ and $z_{m,n}\to z$,
then $\Pr(Z_{m,n}^+\geq z_{m,n})\to \exp(-2z^2)$. Suppose further that
$m/n\to 1$ as $n\to\infty$.  Then $\sqrt{mn/(m+n)}\sim \sqrt{n/2}$.
A question then is whether there exists a function $g(z)$ such that
\beq\label{asymptexp}
\Pr\left(Z_{m,n}^+ \geq z_{m,n}\right) = \exp(-2z^2)
\left(1+\frac{g(z)}{\sqrt{n}} + o\left(\frac 1{\sqrt{n}}\right)\right).
\eeq
Hodges \cite[pp.\ 475-476,481]{Hodges} shows that no such function
$g$ exists. Rather than a $o(1/\sqrt{n})$ error, there is an
``oscillatory'' term which is only $O(1/\sqrt{n})$.
Hodges considers $n=m+2$ (with our convention that $n\geq m$).

If $m=n$, successive possible values of $F_m-G_n$ differ by $1/n$, and
values of $Z_{m,n}^+$ (or our $M$) by $1/\sqrt{2n}$. Thus for fixed
$z$, which are of interest in finding critical values, $z_{n,n}$  can 
only converge to $z$ at a $O(1/\sqrt{n})$ rate.
It seems (to us) unreasonable then to expect (\ref{asymptexp}) to
hold. For $n=m+2$, successive possible values of $F_m-G_n$ typically
(although not always) differ by at most $4/(n(n-2))$, and possible
values of $Z_{m,n}^+$ by $O(n^{-3/2})$, so $z_{m,n}$ can converge
to $z$ at that rate. Then (\ref{asymptexp}) is more plausible 
and it is of interest that Hodges showed it fails.

Here are numerical examples for $m=n-1$, so $L_{m,n}=n(n-1)$, and for
$D_{m,n}$ rather than $Z_{m,n}^+$. We focus on critical values $k$
and $d=k/(n(n-1))$ at the $0.05$ level, having $p$-values $pv$ a little
less than $0.05$. Let $reler$ be the relative error of $dkwm$ as an
approximation to $pv$. By analogy with (\ref{asymptexp}), let us
see how $\sqrt{n}\cdot reler$ behaves.

\begin{center}
\small{
\tablecaption{Behavior of the relative error of $dkwm$ for $m=n-1$}
\tablehead{\bfseries  $n$ &  $k$\ \ &\ \ \  $pv$   & $reler$ 
&$\sqrt{n}\cdot reler$ &
& $n$ &$k\ \ \ $ &\ \ \  $pv$ & $reler$ &$\sqrt{n}\cdot reler$ \\ \hline} 
\tabletail{\hline \multicolumn{11}{r}{\emph{Continued on next page}}\\}
\tablelasttail{\hline} \label{tab:asympt}
\begin{supertabular}{rrlllcrrlll}
 40 &457 &.04968 &.3579 &2.264 &\   &400 &15066 &.049986 &.1379 &2.758 \\
100 &1850 &.049985 &.2395 &2.395 & &500 &21216 &.049983 &.08052 &1.800 \\
200 &5302 &.049885 &.1627 &2.301 & &600 &27889 &.049984 &.08250 &2.021 \\
300 &9771 &.049995 &.1448 &2.507 & & & & & & \\
 \end{supertabular} }%
\end{center}
\vskip 1cm
Here the numbers $\sqrt{n}\cdot reler$ also seem ``oscillatory'' rather
than tending to a constant.

Hodges' argument suggests that the approximation $S_{\textrm{pli}}$,
or any approximation implying an asymptotic expansion,
cannot improve on the $O(1/\sqrt{n})$ order of the relative
error of the simple asymptotic approximation $\beta(M)$;
it may often (but not always, e.g.\ for $m=n$) give smaller multiples 
of $1/\sqrt{n}$, but not $o(1/\sqrt{n})$.

\newpage

\appendix
\section{Details for $m=n\leq 458$}
Here we give details on $\delta_n$ as in Theorem \ref{meqn}(e),
giving data to show by how much (\ref{mt}) fails when $n\leq 457$. 

Recall that for $m=n$, we define $M = k/\sqrt{2n}$. For each $1\leq n \leq 457$, 
we define $k_{\max}$ to be the $k$ such that $1\leq k \leq n$ and 
$
\dfrac{P_{n,n,M}}{2 e^{-2M^2}}
$ is the largest. Since (\ref{mt}) fails for $n\leq 457$, when plugging in $k = k_{\max}$, we must have $$
\dfrac{P_{n,n,M}}{2 e^{-2M^2}} > 1.
$$

Define 
$$
\delta_n := \dfrac{P_{n,n,M}}{2 e^{-2M^2}}-1,
$$ 
where $M = k_{\max}/\sqrt{2n}$.
Then for any fixed $n\leq 457$ and $M > 0$,
$$
P_{n,n,M}=\text{Pr}\left(KS_{n,n} \geq M  \right)\leq 2(1+\delta_n) e^{-2M^2}.
$$
When $n$ increases, the general trend of $\delta_n$ is to decrease, but 
$\delta_n$ is not strictly decreasing, e.g.\ from $n=7$ to $n=8$
(Table \ref{tab:table2}).
For $N\leq 457$, we define $$\Delta_N = \max\{\delta_n\st N \leq n \leq 457\}.$$ Then it is clear that for all $n \geq N$ and $M>0$,
\beq
P_{n,n,M}=\text{Pr}\left(KS_{n,n} \geq M  \right)\leq 2(1+\Delta_N) e^{-2M^2}.
\eeq

In Table \ref{tab:table1} 
we list some pairs $(N, \Delta_N)$ for $1\leq N \leq 455$. The values 
of $\delta_n$ and $\Delta_N$ were originally output
 by Mathematica rounded to 5 decimal places. 
We added $.00001$ to the rounded numbers to assure getting upper bounds.

\begin{center}
\small{
\tablecaption{Selected Pairs $(N, \Delta_N)$}
\tablehead{\bfseries  $N$     &  $\Delta_N$&   & $N$  &  $\Delta_N$ &  &  $N$ & $\Delta_N$ \\ \hline}
\tabletail{\hline \multicolumn{8}{r}{\emph{Continued on next page}}\\}
\tablelasttail{\hline} \label{tab:table1}
\begin{supertabular}{cccccccc}

   1     &  0.35915  &       & 75    & 0.00276  &      & 215   & 0.00045  \\
    2     & 0.23152  &       & 80    & 0.00234  &       & 225   & 0.00041  \\
    3     & 0.13811  &       & 85    & 0.00229  &       & 230   & 0.00039  \\
    4     & 0.08432  &       & 90    & 0.00203  &       & 235   & 0.00036  \\
    5     & 0.08030  &       & 95    & 0.00192  &       & 240   & 0.00034  \\
    6     & 0.06223  &       & 100   & 0.00177  &       & 250   & 0.00032  \\
    7     & 0.04287  &       & 105   & 0.00160  &       & 255   & 0.00028  \\
    9     & 0.04048  &       & 110   & 0.00155  &       & 265   & 0.00028  \\
    10    & 0.03401  &       & 115   & 0.00136  &       & 270   & 0.00026  \\
    11    & 0.02629  &       & 120   & 0.00133  &       & 275   & 0.00024  \\
    13    & 0.02603  &       & 125   & 0.00124  &       & 285   & 0.00023  \\
    14    & 0.02376  &       & 130   & 0.00112  &       & 290   & 0.00020  \\
    15    & 0.02065  &       & 135   & 0.00111  &       & 305   & 0.00018  \\
    16    & 0.01773  &       & 140   & 0.00101  &       & 310   & 0.00016  \\
    18    & 0.01755  &       & 145   & 0.00095  &       & 325   & 0.00015  \\
    20    & 0.01511  &       & 150   & 0.00092  &       & 330   & 0.00013  \\
    24    & 0.01237  &       & 155   & 0.00083  &       & 345   & 0.00012  \\
    28    & 0.00923  &       & 160   & 0.00080  &       & 350   & 0.00011  \\
    32    & 0.00865  &       & 165   & 0.00078  &       & 355   & 0.00010  \\
    36    & 0.00707  &       & 170   & 0.00070  &       & 365   & 0.00009  \\
    40    & 0.00645  &       & 175   & 0.00068  &       & 370   & 0.00008  \\
    44    & 0.00549  &       & 180   & 0.00066  &       & 375   & 0.00007  \\
    48    & 0.00509  &       & 185   & 0.00060  &       & 390   & 0.00006  \\
    52    & 0.00433  &       & 190   & 0.00058  &       & 395   & 0.00005  \\
    56    & 0.00415  &       & 195   & 0.00056  &       & 415   & 0.00004  \\
    60    & 0.00348  &       & 200   & 0.00052  &       & 420   & 0.00003  \\
    65    & 0.00338  &       & 205   & 0.00048  &       & 440   & 0.00002  \\
    70    & 0.00280  &       & 210   & 0.00048  &       & 455   & 0.00001\\
  %

 \end{supertabular} }%
\end{center}

\vskip0.5cm

For $451 \leq N\leq 458$, values of $\Delta_N$ which are more precise than those
Mathematica displays (it gives just 5 decimal places) are as follows.
In all these cases $k=35$. For $N=458$, $k=36$ would give a still
more negative value. Theorem \ref{meqn}(c) shows that no 
$k$ would give $\Delta_{N}>0$ for any $N\geq 458$. 
\vskip0.1cm

\begin{minipage}[c]{3.7cm}\small{
\begin{tabbing}
Ncol \ \ \  \=\ \ $-7.184\cdot 10^{-7}$  \kill
 \ $N$  \ \ \ \=\ \ \ \ \ $\Delta_N$ \\
 451 \ \>\ $5.116\cdot 10^{-6}$\\ 
 452 \ \>\ $4.707\cdot 10^{-6}$\\ 
 453 \ \>\ $4.156\cdot 10^{-6}$\\ 
 454 \ \>\ $3.462\cdot 10^{-6}$\\ 
 455 \ \>\ $2.627\cdot 10^{-6}$\\ 
 456 \ \>\ $1.649\cdot 10^{-6}$\\ 
 457 \ \>\ $5.309\cdot 10^{-7}$\\ 
 458 \ \> $-7.284\cdot 10^{-7}$\\ 
\end{tabbing}}
\end{minipage}


Recall that for $n\geq 458$, we have $\delta_n\leq 0$.
As stated in Theorem \ref{meqn}(e) we have that for $12\leq n\leq  457$, 
\beq \delta_n < -\dfrac{0.07}{ n} + \dfrac{40}{n^2} -\dfrac{400}{n^3}. \label{dnub}
\eeq
(More precisely,  (\ref{dnub}) should be read as: the Mathematica output 
$\delta_n$ plus $0.00001$ is smaller than the right hand side
 of (\ref{dnub})  when $11<n<458$.)
The formula was found by regression and experimentation. 
In Table \ref{tab:table2}, we provide the values of $\delta_n$ when $1\leq n\leq 11$.


\begin{center}
\tablecaption{$\delta_n$ for $n\leq 11$}
\tablehead{\bfseries  $n$   &  $\delta_n$\footnotemark[1]  & & $n$ & $\delta_n$\footnotemark[1] \\ \hline}
\tabletail{\hline \multicolumn{5}{r}{\emph{Continued on next page}}\\}
\tablelasttail{\hline}
\small{
\begin{supertabular}{rlcrl} 
  1   &        0.35914  & &7     &        0.04286 \\
    2     &        0.23151 & & 8     &        0.04434 \\
    3     &       0.1381 & & 9     &        0.04047 \\
    4     &        0.08431 & &  10    & 0.034 \\
    5     &        0.08029 & &　11    &       0.02628 \\  
    6     &        0.06222 & & &\\  
    \end{supertabular}\label{tab:table2}}
\footnotetext[1]{The data shown in Table \ref{tab:table2} are the Mathematica output without adding $0.00001$.}
\end{center}%

\newpage

\section{Tables for $m < n$}
First, we give Table \ref{tab:D1} for $3\leq m\leq 99$ and $m<n\leq 200$,
showing the $n$ for which the largest $r_{\max}$ is attained, which
is always $n=2m$, the $d_{\max}=k_{\max}/n$ at which $r_{\max}$ is
attained, and ``pvatmax,'' the $p$-value in the numerator of
$r_{\max}$.
In this range, the bound (\ref{useupbds}) was used 
($d_0(m,n)\leq 1/2$ is defined) only for
$95\leq m\leq 99$, to avoid probabilities less than $10^{-14}$ from
the inside method. The given $r_{\max}$ are confirmed.
Details are in Table \ref{tab:D2}, first 5 rows, last
2 columns.



\begin{center}
\tablecaption{$3\leq m\leq 99$, $m<n \leq 200$.}
\tablehead{\bfseries  $m$     & $n$     & $r_{\max}$  & $k_{\max}$  &  pvatmax & $d_{\max}$   \\ \hline}
\tabletail{\hline \multicolumn{6}{r}{\emph{Continued on next page}}\\}
\tablelasttail{\hline}
{\small
\begin{supertabular}{rrlrll}\small
    3     & 6     & 0.986116 & 4     & 0.333333 & 0.666667 \\
    4     & 8     & 0.973325 & 4     & 0.513131 & 0.5   \\
    5     & 10    & 0.951143 & 4     & 0.654679 & 0.4   \\
    6     & 12    & 0.938437 & 5     & 0.468003 & 0.416667 \\
    7     & 14    & 0.947585 & 6     & 0.341305 & 0.428571 \\
    8     & 16    & 0.950533 & 6     & 0.424185 & 0.375 \\
    9     & 18    & 0.949182 & 6     & 0.500403 & 0.333333 \\
    10    & 20    & 0.944748 & 6     & 0.569105 & 0.3   \\
    11    & 22    & 0.946271 & 7     & 0.42873 & 0.318182 \\
    12    & 24    & 0.946955 & 8     & 0.320096 & 0.333333 \\
    13    & 26    & 0.949675 & 8     & 0.368058 & 0.307692 \\
    14    & 28    & 0.950815 & 8     & 0.414328 & 0.285714 \\
    15    & 30    & 0.950668 & 8     & 0.458559 & 0.266667 \\
    16    & 32    & 0.950333 & 9     & 0.351588 & 0.28125 \\
    17    & 34    & 0.951642 & 9     & 0.388814 & 0.264706 \\
    18    & 36    & 0.952087 & 9     & 0.424878 & 0.25  \\
    19    & 38    & 0.9527 & 10    & 0.32966 & 0.263158 \\
    20    & 40    & 0.953956 & 10    & 0.360358 & 0.25  \\
    21    & 42    & 0.954631 & 10    & 0.390399 & 0.238095 \\
    22    & 44    & 0.954788 & 10    & 0.419677 & 0.227273 \\
    23    & 46    & 0.95505 & 11    & 0.330725 & 0.23913 \\
    24    & 48    & 0.955966 & 11    & 0.356137 & 0.229167 \\
    25    & 50    & 0.956499 & 11    & 0.381112 & 0.22  \\
    26    & 52    & 0.956683 & 11    & 0.405588 & 0.211538 \\
    27    & 54    & 0.957278 & 12    & 0.323585 & 0.222222 \\
    28    & 56    & 0.958022 & 12    & 0.345065 & 0.214286 \\
    29    & 58    & 0.958501 & 12    & 0.366261 & 0.206897 \\
    30    & 60    & 0.958735 & 12    & 0.387131 & 0.2   \\
    31    & 62    & 0.958918 & 13    & 0.311609 & 0.209677 \\
    32    & 64    & 0.959602 & 13    & 0.330051 & 0.203125 \\
    33    & 66    & 0.960091 & 13    & 0.348314 & 0.19697 \\
    34    & 68    & 0.960399 & 13    & 0.366366 & 0.191176 \\
    35    & 70    & 0.960536 & 13    & 0.384182 & 0.185714 \\
    36    & 72    & 0.961028 & 14    & 0.313042 & 0.194444 \\
    37    & 74    & 0.961533 & 14    & 0.328951 & 0.189189 \\
    38    & 76    & 0.9619 & 14    & 0.344729 & 0.184211 \\
    39    & 78    & 0.962136 & 14    & 0.360355 & 0.179487 \\
    40    & 80    & 0.962249 & 14    & 0.375811 & 0.175 \\
    41    & 82    & 0.962708 & 15    & 0.309089 & 0.182927 \\
    42    & 84    & 0.963123 & 15    & 0.322988 & 0.178571 \\
    43    & 86    & 0.963437 & 15    & 0.336793 & 0.174419 \\
    44    & 88    & 0.963654 & 15    & 0.350491 & 0.170455 \\
    45    & 90    & 0.963776 & 15    & 0.364068 & 0.166667 \\
    46    & 92    & 0.964152 & 16    & 0.301667 & 0.173913 \\
    47    & 94    & 0.964521 & 16    & 0.313932 & 0.170213 \\
    48    & 96    & 0.964812 & 16    & 0.326132 & 0.166667 \\
    49    & 98    & 0.965027 & 16    & 0.338257 & 0.163265 \\
    50    & 100   & 0.965171 & 16    & 0.350299 & 0.16  \\
    51    & 102   & 0.965387 & 17    & 0.29201 & 0.166667 \\
    52    & 104   & 0.965731 & 17    & 0.30292 & 0.163462 \\
    53    & 106   & 0.966015 & 17    & 0.313788 & 0.160377 \\
    54    & 108   & 0.966239 & 17    & 0.324605 & 0.157407 \\
    55    & 110   & 0.966407 & 17    & 0.335364 & 0.154545 \\
    56    & 112   & 0.966519 & 17    & 0.346059 & 0.151786 \\
    57    & 114   & 0.966794 & 18    & 0.29073 & 0.157895 \\
    58    & 116   & 0.967076 & 18    & 0.300472 & 0.155172 \\
    59    & 118   & 0.967311 & 18    & 0.310182 & 0.152542 \\
    60    & 120   & 0.9675 & 18    & 0.319853 & 0.15  \\
    61    & 122   & 0.967645 & 18    & 0.329482 & 0.147541 \\
    62    & 124   & 0.967746 & 18    & 0.339061 & 0.145161 \\
    63    & 126   & 0.968 & 19    & 0.286669 & 0.150794 \\
    64    & 128   & 0.968245 & 19    & 0.295428 & 0.148438 \\
    65    & 130   & 0.968453 & 19    & 0.304163 & 0.146154 \\
    66    & 132   & 0.968624 & 19    & 0.312871 & 0.143939 \\
    67    & 134   & 0.96876 & 19    & 0.321547 & 0.141791 \\
    68    & 136   & 0.968862 & 19    & 0.330188 & 0.139706 \\
    69    & 138   & 0.969058 & 20    & 0.280649 & 0.144928 \\
    70    & 140   & 0.96928 & 20    & 0.28857 & 0.142857 \\
    71    & 142   & 0.969473 & 20    & 0.296476 & 0.140845 \\
    72    & 144   & 0.969636 & 20    & 0.304361 & 0.138889 \\
    73    & 146   & 0.96977 & 20    & 0.312224 & 0.136986 \\
    74    & 148   & 0.969876 & 20    & 0.320062 & 0.135135 \\
    75    & 150   & 0.969993 & 21    & 0.273263 & 0.14  \\
    76    & 152   & 0.970201 & 21    & 0.280462 & 0.138158 \\
    77    & 154   & 0.970385 & 21    & 0.287651 & 0.136364 \\
    78    & 156   & 0.970544 & 21    & 0.294827 & 0.134615 \\
    79    & 158   & 0.970681 & 21    & 0.301987 & 0.132911 \\
    80    & 160   & 0.970794 & 21    & 0.30913 & 0.13125 \\
    81    & 162   & 0.970884 & 21    & 0.316252 & 0.12963 \\
    82    & 164   & 0.971022 & 22    & 0.271515 & 0.134146 \\
    83    & 166   & 0.971201 & 22    & 0.278079 & 0.13253 \\
    84    & 168   & 0.97136 & 22    & 0.284636 & 0.130952 \\
    85    & 170   & 0.9715 & 22    & 0.291182 & 0.129412 \\
    86    & 172   & 0.97162 & 22    & 0.297717 & 0.127907 \\
    87    & 174   & 0.971721 & 22    & 0.304238 & 0.126437 \\
    88    & 176   & 0.971804 & 22    & 0.310744 & 0.125 \\
    89    & 178   & 0.971931 & 23    & 0.268046 & 0.129213 \\
    90    & 180   & 0.972091 & 23    & 0.274057 & 0.127778 \\
    91    & 182   & 0.972234 & 23    & 0.280063 & 0.126374 \\
    92    & 184   & 0.972361 & 23    & 0.286062 & 0.125 \\
    93    & 186   & 0.972472 & 23    & 0.292052 & 0.123656 \\
    94    & 188   & 0.972567 & 23    & 0.298032 & 0.12234 \\
    95    & 190   & 0.972647 & 23    & 0.304 & 0.121053 \\
    96    & 192   & 0.972743 & 24    & 0.263293 & 0.125 \\
    97    & 194   & 0.97289 & 24    & 0.268818 & 0.123711 \\
    98    & 196   & 0.973022 & 24    & 0.274341 & 0.122449 \\
    99    & 198   & 0.973142 & 24    & 0.279858 & 0.121212 \\
    \end{supertabular}
  \label{tab:D1}}%
\end{center}
\

Next, for each $m$ with $100\leq m\leq 199$ we searched by computer
among all $n=m+1,\dotsc,200$. For each such $n$, $r_{\max}(m,n)$ was
found, and then for given $m$, the largest such $r_{\max}$, called
$rmaxx$ in Table \ref{tab:HTH}, attained at $n=n_{\max}$ and for that
$n$, at $d=$ dmaxx $= k_{\max}/L_{m,n_{\max}}$ (recall that
$L_{m,n}$ is the least common multiple of $m$ and $n$), 
and with a $p$-value ``pvatmax''
in the numerator of $rmaxx$. There are columns in Table \ref{tab:HTH}
for each of these.

For each $m<n\leq 200$ and each possible value $d$ of $D_{m,n}$ in the
range (\ref{range}) where the $p$-value by the inside method was
found to be less than $10^{-14}$ and so would have too few reliable 
significant digits, we evaluated instead the upper bound
$pv_{ub}(m,n,d)$ as in Theorem \ref{oneandtwoside}(a) and took
the ratio 
\beq\label{rub}
r_{ub}(m,n,d) = pv_{ub}(m,n,d)/\left(2\exp(-2M^2)\right)
\eeq
where as usual $M = \sqrt{mn/(m+n)}d$. We took the maximum of these
for the possible values of $d$ and the ratio of that maximum to
$r_{\max}(m,n)$ as evaluated for all other possible values of $d$.
Then we took in turn the maximum of all such ratios for fixed $m$
over $n$ with $m<n\leq 200$, giving $mrmr$ (``maximum ratio of
maximum ratios'') in the last column of Table \ref{tab:HTH}.
As all these are less than 1 (the largest, for $m= 196$, is less
than $0.415$), we confirm that $r_{\max}(m,n)$ is not attained in
the range (\ref{range}) for $100\leq m<n\leq 200$ and so
the given values of $n_{\max}$ and $rmaxx$ are confirmed.

For given $m$, $mrmr$ often, but not always, occurs when
$n=n_{\max}$. For example, it does when $m=132$ and for
$195\leq m\leq 199$, but not for $m=168$, for which
$n_{\max} = 196$ but $mrmr$ occurs for $n=169$.

In Tables \ref{tab:D1} and \ref{tab:D2} the ratio $n/m$ is always 2,
in Table \ref{tab:D1} and for $m=100$ because $n_{\max}=2m$ from the 
computer search, and in Table \ref{tab:D2} by our choice. In the range
$101\leq m < n\leq 200$, $n_{\max}/m=2$ is not possible, but
$3/2$ is and occurs as described in Fact \ref{mneqn}(d). 
For example, when
$m=175$, $n_{\max}=176$, even though $n=200$ would have given a simpler
ratio $n/m=8/7$; but $r_{\max}(175,200)= 0.927656 < 0.928771 = r_{\max}(175,176)$.
Ratios occur of $n_{\max}/m = 9/7 = 198/154$, $10/7 = 190/133$, and
$11/7 = 187/119$.

\newpage

%
\begin{center}
\tablecaption{$100\leq m < n\leq 200$}
\tablehead{\bfseries  $m$     &  $n_{\max}$   &   $rmaxx$  &  $k_{\max}$ & pvatmax  &  dmaxx & $d_0(m,200)$ & $mrmr$  \\ \hline}
\tabletail{\hline \multicolumn{8}{c}{\emph{Continued on next page}}\\}
\tablelasttail{\hline}
\small{
    \begin{supertabular}{rrlrllll}
    100   & 200   & 0.973248 & 24    & 0.28537 & 0.12  & 0.49  & 0.238509 \\
    101   & 200   & 0.913382 & 2134  & 0.408438 & 0.105644 & 0.482525 & 0.228132 \\
    102   & 153   & 0.943929 & 36    & 0.346915 & 0.117647 & 0.480784 & 0.215796 \\
    103   & 155   & 0.913333 & 1764  & 0.403162 & 0.110492 & 0.479951 & 0.211469 \\
    104   & 156   & 0.944382 & 36    & 0.358576 & 0.115385 & 0.478846 & 0.214312 \\
    105   & 175   & 0.93144 & 58    & 0.375393 & 0.110476 & 0.477143 & 0.216784 \\
    106   & 159   & 0.944769 & 37    & 0.337672 & 0.116352 & 0.475377 & 0.220575 \\
    107   & 161   & 0.914677 & 1886  & 0.391834 & 0.109479 & 0.474439 & 0.220863 \\
    108   & 162   & 0.945233 & 37    & 0.348785 & 0.114198 & 0.473148 & 0.226247 \\
    109   & 164   & 0.915258 & 1921  & 0.403431 & 0.107463 & 0.471606 & 0.226013 \\
    110   & 165   & 0.94563 & 37    & 0.35982 & 0.112121 & 0.470909 & 0.228994 \\
    111   & 185   & 0.932974 & 60    & 0.36867 & 0.108108 & 0.46973 & 0.235811 \\
    112   & 168   & 0.946023 & 38    & 0.339124 & 0.113095 & 0.468214 & 0.235084 \\
    113   & 170   & 0.916523 & 2048  & 0.391779 & 0.106611 & 0.466504 & 0.236755 \\
    114   & 171   & 0.946435 & 38    & 0.34966 & 0.111111 & 0.465702 & 0.245198 \\
    115   & 184   & 0.924245 & 96    & 0.395831 & 0.104348 & 0.464565 & 0.245341 \\
    116   & 174   & 0.946787 & 38    & 0.360125 & 0.109195 & 0.462931 & 0.246682 \\
    117   & 195   & 0.934419 & 61    & 0.381039 & 0.104274 & 0.461538 & 0.249586 \\
    118   & 177   & 0.947179 & 39    & 0.339676 & 0.110169 & 0.460593 & 0.256402 \\
    119   & 187   & 0.92098 & 134   & 0.40119 & 0.102368 & 0.459328 & 0.256227 \\
    120   & 180   & 0.947549 & 39    & 0.349682 & 0.108333 & 0.46  & 0.257563 \\
    121   & 182   & 0.918795 & 2314  & 0.369177 & 0.105077 & 0.457107 & 0.260102 \\
    122   & 183   & 0.94787 & 40    & 0.329881 & 0.10929 & 0.455984 & 0.266913 \\
    123   & 164   & 0.935287 & 53    & 0.366045 & 0.107724 & 0.454878 & 0.265827 \\
    124   & 186   & 0.948254 & 40    & 0.339454 & 0.107527 & 0.453871 & 0.267777 \\
    125   & 200   & 0.926795 & 101   & 0.385868 & 0.101 & 0.454 & 0.269952 \\
    126   & 189   & 0.94859 & 40    & 0.348975 & 0.10582 & 0.451667 & 0.276748 \\
    127   & 191   & 0.92039 & 2493  & 0.367425 & 0.102774 & 0.450827 & 0.276017 \\
    128   & 192   & 0.948905 & 41    & 0.329447 & 0.106771 & 0.449688 & 0.27736 \\
    129   & 172   & 0.936549 & 54    & 0.372684 & 0.104651 & 0.448721 & 0.279215 \\
    130   & 195   & 0.949257 & 41    & 0.338568 & 0.105128 & 0.447692 & 0.285965 \\
    131   & 197   & 0.921385 & 2571  & 0.38654 & 0.099624 & 0.446565 & 0.287611 \\
    132   & 198   & 0.949565 & 41    & 0.347641 & 0.103535 & 0.445455 & 0.294401 \\
    133   & 190   & 0.923341 & 132   & 0.395393 & 0.099248 & 0.444436 & 0.293273 \\
    134   & 135   & 0.920683 & 2045  & 0.330121 & 0.113046 & 0.443955 & 0.280389 \\
    135   & 180   & 0.937714 & 56    & 0.356856 & 0.103704 & 0.442963 & 0.274898 \\
    136   & 170   & 0.930667 & 70    & 0.375306 & 0.102941 & 0.441765 & 0.273921 \\
    137   & 138   & 0.921316 & 2091  & 0.342759 & 0.1106 & 0.440766 & 0.277844 \\
    138   & 184   & 0.93829 & 56    & 0.370191 & 0.101449 & 0.44  & 0.284898 \\
    139   & 140   & 0.921695 & 2121  & 0.351497 & 0.108993 & 0.439101 & 0.279904 \\
    140   & 175   & 0.931495 & 71    & 0.376012 & 0.101429 & 0.438571 & 0.283698 \\
    141   & 188   & 0.938842 & 57    & 0.362092 & 0.101064 & 0.437518 & 0.290272 \\
    142   & 143   & 0.922434 & 2310  & 0.291798 & 0.113759 & 0.436549 & 0.294849 \\
    143   & 144   & 0.922679 & 2326  & 0.295749 & 0.112956 & 0.436189 & 0.302968 \\
    144   & 192   & 0.939363 & 58    & 0.354142 & 0.100694 & 0.435 & 0.295917 \\
    145   & 174   & 0.92777 & 87    & 0.381501 & 0.1   & 0.433966 & 0.296235 \\
    146   & 147   & 0.92338 & 2375  & 0.307108 & 0.110661 & 0.433151 & 0.304232 \\
    147   & 196   & 0.939886 & 58    & 0.366614 & 0.098639 & 0.432347 & 0.30574 \\
    148   & 185   & 0.933056 & 73    & 0.376649 & 0.098649 & 0.431622 & 0.302085 \\
    149   & 150   & 0.924015 & 2423  & 0.318878 & 0.108412 & 0.43104 & 0.305384 \\
    150   & 200   & 0.940395 & 59    & 0.358464 & 0.098333 & 0.431667 & 0.313118 \\
    151   & 152   & 0.9244 & 2455  & 0.326689 & 0.106962 & 0.42947 & 0.320836 \\
    152   & 190   & 0.933791 & 74    & 0.376618 & 0.097368 & 0.428684 & 0.307194 \\
    153   & 154   & 0.924759 & 2488  & 0.334009 & 0.105594 & 0.427876 & 0.31393 \\
    154   & 198   & 0.926355 & 132   & 0.384897 & 0.095238 & 0.427403 & 0.321538 \\
    155   & 186   & 0.929738 & 90    & 0.381638 & 0.096774 & 0.426452 & 0.329063 \\
    156   & 195   & 0.934499 & 75    & 0.376378 & 0.096154 & 0.425769 & 0.314584 \\
    157   & 158   & 0.925501 & 2711  & 0.282122 & 0.109288 & 0.424841 & 0.322061 \\
    158   & 159   & 0.925721 & 2728  & 0.285641 & 0.10859 & 0.424557 & 0.329455 \\
    159   & 160   & 0.925934 & 2745  & 0.289158 & 0.107901 & 0.423459 & 0.328888 \\
    160   & 200   & 0.935183 & 76    & 0.375946 & 0.095 & 0.42375 & 0.339848 \\
    161   & 162   & 0.926347 & 2780  & 0.29579 & 0.106587 & 0.422174 & 0.329698 \\
    162   & 189   & 0.927765 & 108   & 0.38127 & 0.095238 & 0.421481 & 0.336907 \\
    163   & 164   & 0.92674 & 2814  & 0.302799 & 0.105267 & 0.420798 & 0.344069 \\
    164   & 165   & 0.926928 & 2831  & 0.306296 & 0.104619 & 0.420366 & 0.341805 \\
    165   & 198   & 0.931538 & 93    & 0.380526 & 0.093939 & 0.419697 & 0.337482 \\
    166   & 167   & 0.927286 & 2865  & 0.313277 & 0.103348 & 0.418855 & 0.343963 \\
    167   & 168   & 0.927455 & 2882  & 0.316759 & 0.102723 & 0.417725 & 0.350977 \\
    168   & 196   & 0.928852 & 110   & 0.381517 & 0.093537 & 0.417619 & 0.357974 \\
    169   & 170   & 0.927778 & 2917  & 0.323319 & 0.101532 & 0.416746 & 0.343789 \\
    170   & 171   & 0.927934 & 2934  & 0.326785 & 0.100929 & 0.416471 & 0.35065 \\
    171   & 172   & 0.928084 & 2951  & 0.330246 & 0.100333 & 0.415351 & 0.35752 \\
    172   & 173   & 0.928229 & 2968  & 0.333699 & 0.099745 & 0.415116 & 0.364343 \\
    173   & 174   & 0.928384 & 3160  & 0.274412 & 0.104976 & 0.414451 & 0.352725 \\
    174   & 175   & 0.92858 & 3178  & 0.277564 & 0.104368 & 0.413966 & 0.356959 \\
    175   & 176   & 0.928771 & 3196  & 0.280715 & 0.103766 & 0.413571 & 0.363665 \\
    176   & 177   & 0.928956 & 3214  & 0.283863 & 0.103172 & 0.4125 & 0.370297 \\
    177   & 178   & 0.929141 & 3233  & 0.286679 & 0.102615 & 0.41209 & 0.376845 \\
    178   & 179   & 0.929321 & 3251  & 0.289823 & 0.102034 & 0.411629 & 0.383179 \\
    179   & 180   & 0.929496 & 3269  & 0.292965 & 0.101459 & 0.410698 & 0.369441 \\
    180   & 181   & 0.929666 & 3287  & 0.296104 & 0.10089 & 0.410556 & 0.375911 \\
    181   & 182   & 0.929831 & 3305  & 0.299239 & 0.100328 & 0.409309 & 0.382325 \\
    182   & 183   & 0.929992 & 3323  & 0.302371 & 0.099772 & 0.408901 & 0.388746 \\
    183   & 184   & 0.930148 & 3341  & 0.3055 & 0.099222 & 0.408087 & 0.374912 \\
    184   & 185   & 0.930299 & 3359  & 0.308624 & 0.098678 & 0.407826 & 0.381228 \\
    185   & 186   & 0.930446 & 3378  & 0.311415 & 0.098169 & 0.407027 & 0.387516 \\
    186   & 187   & 0.930591 & 3396  & 0.314533 & 0.097637 & 0.40672 & 0.393782 \\
    187   & 188   & 0.930732 & 3414  & 0.317646 & 0.09711 & 0.406043 & 0.387575 \\
    188   & 189   & 0.930867 & 3432  & 0.320755 & 0.096589 & 0.405426 & 0.386262 \\
    189   & 190   & 0.930999 & 3450  & 0.323859 & 0.096074 & 0.404894 & 0.39243 \\
    190   & 191   & 0.931125 & 3468  & 0.326959 & 0.095564 & 0.404474 & 0.398548 \\
    191   & 192   & 0.931267 & 3679  & 0.271066 & 0.100322 & 0.403953 & 0.404579 \\
    192   & 193   & 0.931438 & 3699  & 0.27362 & 0.099822 & 0.403542 & 0.392781 \\
    193   & 194   & 0.931607 & 3718  & 0.276457 & 0.0993 & 0.402409 & 0.397034 \\
    194   & 195   & 0.931772 & 3737  & 0.279293 & 0.098784 & 0.402165 & 0.402986 \\
    195   & 196   & 0.931932 & 3756  & 0.282127 & 0.098273 & 0.402051 & 0.408865 \\
    196   & 197   & 0.932089 & 3775  & 0.284959 & 0.097768 & 0.400408 & 0.414765 \\
    197   & 198   & 0.932242 & 3794  & 0.287789 & 0.097267 & 0.400533 & 0.401356 \\
    198   & 199   & 0.932391 & 3813  & 0.290616 & 0.096772 & 0.401162 & 0.407172 \\
    199   & 200   & 0.932536 & 3832  & 0.293442 & 0.096281 & 0.398719 & 0.412943 \\
 \end{supertabular} }%
\label{tab:HTH}
\end{center}
\

The following Table \ref{tab:D2} treats $95\leq m\leq 300$ and $n=2m$. In each such
case, $r_{\max}(m,n)$ was computed. It has a numerator $p$-value
``pvatmax'' attained at $d_{\max} = k_{\max}/n$.

Throughout the table, $r_{\max}$ continues to increase, as it does in
Table \ref{tab:D1} for $m\geq 16$, and as stated in Fact \ref{doubles}(a).

In the last column, $rbd_{\max}$ is the maximum of $r_{ub}(m,2m,d)$ as defined in (\ref{rub}) for $d$ in the range 
(\ref{range}). These $rbd_{\max}$ tend
to increase with $m$, although not monotonically. All values shown are
less than $0.65$, which is less than $r_{\max}$ for all the values
of $m$ shown. This confirms the values of $r_{\max}$.
\begin{center}
{\small
\tablecaption{$95\leq m\leq 300$, $n=2m$}\label{tab:D2}
\tablehead{\bfseries  $m$     &  $n$   &   $r_{\max}$  &  $k_{\max}$ & pvatmax  &  $d_{\max}$ & $d_0(m,2m)$ & $rbd_{\max}$  \\ \hline}
\tabletail{\hline \multicolumn{8}{r}{\emph{Continued on next page}}\\}
\tablelasttail{\hline}
\begin{supertabular}{rrlrllll}
%
  95    & 190   & 0.972647 & 23    & 0.304 & 0.121053 & 0.5   & 0.221227 \\
    96    & 192   & 0.972743 & 24    & 0.263293 & 0.125 & 0.5   & 0.217684 \\
    97    & 194   & 0.97289 & 24    & 0.268818 & 0.123711 & 0.494845 & 0.22868 \\
    98    & 196   & 0.973022 & 24    & 0.274341 & 0.122449 & 0.494898 & 0.225026 \\
    99    & 198   & 0.973142 & 24    & 0.279858 & 0.121212 & 0.489899 & 0.235886 \\
    100   & 200   & 0.973248 & 24    & 0.28537 & 0.12  & 0.49  & 0.232128 \\
    101   & 202   & 0.973341 & 24    & 0.290874 & 0.118812 & 0.485149 & 0.242848 \\
    102   & 204   & 0.973421 & 24    & 0.296371 & 0.117647 & 0.485294 & 0.238995 \\
    103   & 206   & 0.973488 & 24    & 0.301857 & 0.116505 & 0.480583 & 0.249572 \\
    104   & 208   & 0.973611 & 25    & 0.262685 & 0.120192 & 0.480769 & 0.245632 \\
    105   & 210   & 0.973737 & 25    & 0.267779 & 0.119048 & 0.47619 & 0.256064 \\
    106   & 212   & 0.973852 & 25    & 0.27287 & 0.117925 & 0.476415 & 0.252044 \\
    107   & 214   & 0.973955 & 25    & 0.277958 & 0.116822 & 0.471963 & 0.262329 \\
    108   & 216   & 0.974047 & 25    & 0.283042 & 0.115741 & 0.472222 & 0.258236 \\
    109   & 218   & 0.974129 & 25    & 0.28812 & 0.114679 & 0.472477 & 0.254206 \\
    110   & 220   & 0.974199 & 25    & 0.293191 & 0.113636 & 0.468182 & 0.264215 \\
    111   & 222   & 0.974264 & 26    & 0.255903 & 0.117117 & 0.463964 & 0.274206 \\
    112   & 224   & 0.974386 & 26    & 0.260616 & 0.116071 & 0.464286 & 0.269986 \\
    113   & 226   & 0.974498 & 26    & 0.265329 & 0.115044 & 0.460177 & 0.27983 \\
    114   & 228   & 0.9746 & 26    & 0.270039 & 0.114035 & 0.460526 & 0.275555 \\
    115   & 230   & 0.974692 & 26    & 0.274746 & 0.113043 & 0.456522 & 0.285254 \\
    116   & 232   & 0.974776 & 26    & 0.279451 & 0.112069 & 0.456897 & 0.28093 \\
    117   & 234   & 0.97485 & 26    & 0.284151 & 0.111111 & 0.452991 & 0.290484 \\
    118   & 236   & 0.974915 & 26    & 0.288846 & 0.110169 & 0.449153 & 0.299999 \\
    119   & 238   & 0.974975 & 27    & 0.253039 & 0.113445 & 0.44958 & 0.295527 \\
    120   & 240   & 0.975085 & 27    & 0.25741 & 0.1125 & 0.45  & 0.291118 \\
    121   & 242   & 0.975187 & 27    & 0.261782 & 0.11157 & 0.446281 & 0.300389 \\
    122   & 244   & 0.975281 & 27    & 0.266152 & 0.110656 & 0.442623 & 0.309615 \\
    123   & 246   & 0.975366 & 27    & 0.270521 & 0.109756 & 0.443089 & 0.305077 \\
    124   & 248   & 0.975444 & 27    & 0.274888 & 0.108871 & 0.439516 & 0.314162 \\
    125   & 250   & 0.975514 & 27    & 0.279251 & 0.108 & 0.44  & 0.309596 \\
    126   & 252   & 0.975576 & 27    & 0.283611 & 0.107143 & 0.436508 & 0.318543 \\
    127   & 254   & 0.97563 & 27    & 0.287967 & 0.106299 & 0.437008 & 0.313952 \\
    128   & 256   & 0.975721 & 28    & 0.253321 & 0.109375 & 0.433594 & 0.322764 \\
    129   & 258   & 0.975816 & 28    & 0.257387 & 0.108527 & 0.434109 & 0.318152 \\
    130   & 260   & 0.975904 & 28    & 0.261453 & 0.107692 & 0.430769 & 0.326832 \\
    131   & 262   & 0.975985 & 28    & 0.265518 & 0.10687 & 0.427481 & 0.335454 \\
    132   & 264   & 0.976059 & 28    & 0.269582 & 0.106061 & 0.42803 & 0.330751 \\
    133   & 266   & 0.976126 & 28    & 0.273644 & 0.105263 & 0.424812 & 0.339243 \\
    134   & 268   & 0.976187 & 28    & 0.277703 & 0.104478 & 0.425373 & 0.334527 \\
    135   & 270   & 0.976241 & 28    & 0.28176 & 0.103704 & 0.422222 & 0.342891 \\
    136   & 272   & 0.976302 & 29    & 0.24855 & 0.106618 & 0.422794 & 0.338166 \\
    137   & 274   & 0.976392 & 29    & 0.252341 & 0.105839 & 0.419708 & 0.346406 \\
    138   & 276   & 0.976476 & 29    & 0.256133 & 0.105072 & 0.416667 & 0.354584 \\
    139   & 278   & 0.976553 & 29    & 0.259924 & 0.104317 & 0.417266 & 0.34979 \\
    140   & 280   & 0.976625 & 29    & 0.263715 & 0.103571 & 0.414286 & 0.357847 \\
    141   & 282   & 0.976691 & 29    & 0.267505 & 0.102837 & 0.414894 & 0.35305 \\
    142   & 284   & 0.976752 & 29    & 0.271294 & 0.102113 & 0.411972 & 0.360988 \\
    143   & 286   & 0.976806 & 29    & 0.27508 & 0.101399 & 0.412587 & 0.356191 \\
    144   & 288   & 0.976855 & 29    & 0.278865 & 0.100694 & 0.409722 & 0.364013 \\
    145   & 290   & 0.976921 & 30    & 0.246802 & 0.103448 & 0.406897 & 0.371771 \\
    146   & 292   & 0.977002 & 30    & 0.250345 & 0.10274 & 0.407534 & 0.366924 \\
    147   & 294   & 0.977077 & 30    & 0.253889 & 0.102041 & 0.404762 & 0.37457 \\
    148   & 296   & 0.977148 & 30    & 0.257433 & 0.101351 & 0.405405 & 0.369728 \\
    149   & 298   & 0.977213 & 30    & 0.260976 & 0.100671 & 0.402685 & 0.377264 \\
    150   & 300   & 0.977274 & 30    & 0.264519 & 0.1   & 0.4   & 0.384736 \\
    151   & 302   & 0.97733 & 30    & 0.268061 & 0.099338 & 0.400662 & 0.379856 \\
    152   & 304   & 0.97738 & 30    & 0.271602 & 0.098684 & 0.401316 & 0.375025 \\
    153   & 306   & 0.977426 & 30    & 0.275142 & 0.098039 & 0.398693 & 0.382351 \\
    154   & 308   & 0.977485 & 31    & 0.244214 & 0.100649 & 0.396104 & 0.389613 \\
    155   & 310   & 0.97756 & 31    & 0.247532 & 0.1   & 0.396774 & 0.384751 \\
    156   & 312   & 0.97763 & 31    & 0.250851 & 0.099359 & 0.394231 & 0.391913 \\
    157   & 314   & 0.977695 & 31    & 0.254171 & 0.098726 & 0.39172 & 0.399011 \\
    158   & 316   & 0.977756 & 31    & 0.25749 & 0.098101 & 0.392405 & 0.394124 \\
    159   & 318   & 0.977813 & 31    & 0.26081 & 0.097484 & 0.389937 & 0.401125 \\
    160   & 320   & 0.977865 & 31    & 0.264129 & 0.096875 & 0.390625 & 0.396251 \\
    161   & 322   & 0.977914 & 31    & 0.267447 & 0.096273 & 0.388199 & 0.403157 \\
    162   & 324   & 0.977958 & 31    & 0.270764 & 0.095679 & 0.385802 & 0.41 \\
    163   & 326   & 0.978004 & 32    & 0.240951 & 0.09816 & 0.386503 & 0.40511 \\
    164   & 328   & 0.978074 & 32    & 0.244064 & 0.097561 & 0.384146 & 0.411862 \\
    165   & 330   & 0.978139 & 32    & 0.247179 & 0.09697 & 0.384848 & 0.406986 \\
    166   & 332   & 0.978201 & 32    & 0.250294 & 0.096386 & 0.38253 & 0.413649 \\
    167   & 334   & 0.978259 & 32    & 0.25341 & 0.095808 & 0.38024 & 0.42025 \\
    168   & 336   & 0.978313 & 32    & 0.256526 & 0.095238 & 0.380952 & 0.415365 \\
    169   & 338   & 0.978364 & 32    & 0.259642 & 0.094675 & 0.378698 & 0.421881 \\
    170   & 340   & 0.97841 & 32    & 0.262758 & 0.094118 & 0.376471 & 0.428335 \\
    171   & 342   & 0.978453 & 32    & 0.265873 & 0.093567 & 0.377193 & 0.423445 \\
    172   & 344   & 0.978492 & 32    & 0.268987 & 0.093023 & 0.375 & 0.429816 \\
    173   & 346   & 0.978549 & 33    & 0.240075 & 0.095376 & 0.375723 & 0.424944 \\
    174   & 348   & 0.978611 & 33    & 0.243003 & 0.094828 & 0.373563 & 0.431236 \\
    175   & 350   & 0.97867 & 33    & 0.245932 & 0.094286 & 0.371429 & 0.437467 \\
    176   & 352   & 0.978726 & 33    & 0.248862 & 0.09375 & 0.372159 & 0.432595 \\
    177   & 354   & 0.978778 & 33    & 0.251792 & 0.09322 & 0.370056 & 0.43875 \\
    178   & 356   & 0.978827 & 33    & 0.254723 & 0.092697 & 0.367978 & 0.444844 \\
    179   & 358   & 0.978873 & 33    & 0.257654 & 0.092179 & 0.368715 & 0.439976 \\
    180   & 360   & 0.978915 & 33    & 0.260584 & 0.091667 & 0.366667 & 0.445997 \\
    181   & 362   & 0.978955 & 33    & 0.263514 & 0.09116 & 0.367403 & 0.441149 \\
    182   & 364   & 0.978991 & 33    & 0.266444 & 0.090659 & 0.365385 & 0.447097 \\
    183   & 366   & 0.979048 & 34    & 0.238431 & 0.092896 & 0.363388 & 0.452987 \\
    184   & 368   & 0.979105 & 34    & 0.24119 & 0.092391 & 0.36413 & 0.448147 \\
    185   & 370   & 0.979159 & 34    & 0.243949 & 0.091892 & 0.362162 & 0.453968 \\
    186   & 372   & 0.979211 & 34    & 0.246709 & 0.091398 & 0.362903 & 0.449149 \\
    187   & 374   & 0.979259 & 34    & 0.24947 & 0.090909 & 0.360963 & 0.454903 \\
    188   & 376   & 0.979304 & 34    & 0.252231 & 0.090426 & 0.361702 & 0.450104 \\
    189   & 378   & 0.979347 & 34    & 0.254992 & 0.089947 & 0.357143 & 0.466241 \\
    190   & 380   & 0.979386 & 34    & 0.257753 & 0.089474 & 0.357895 & 0.461424 \\
    191   & 382   & 0.979423 & 34    & 0.260515 & 0.089005 & 0.34555 & 0.508269 \\
    192   & 384   & 0.979457 & 34    & 0.263276 & 0.088542 & 0.34375 & 0.513568 \\
    193   & 386   & 0.97951 & 35    & 0.236154 & 0.090674 & 0.341969 & 0.518807 \\
    194   & 388   & 0.979563 & 35    & 0.238756 & 0.090206 & 0.342784 & 0.513896 \\
    195   & 390   & 0.979613 & 35    & 0.24136 & 0.089744 & 0.341026 & 0.519079 \\
    196   & 392   & 0.979661 & 35    & 0.243964 & 0.089286 & 0.339286 & 0.524203 \\
    197   & 394   & 0.979706 & 35    & 0.246569 & 0.088832 & 0.340102 & 0.51932 \\
    198   & 396   & 0.979749 & 35    & 0.249175 & 0.088384 & 0.338384 & 0.524391 \\
    199   & 398   & 0.979789 & 35    & 0.251781 & 0.08794 & 0.336683 & 0.529404 \\
    200   & 400   & 0.979827 & 35    & 0.254387 & 0.0875 & 0.3375 & 0.52455 \\
    201   & 402   & 0.979862 & 35    & 0.256993 & 0.087065 & 0.335821 & 0.529512 \\
    202   & 404   & 0.979894 & 35    & 0.259599 & 0.086634 & 0.334158 & 0.534418 \\
    203   & 406   & 0.979938 & 36    & 0.233354 & 0.08867 & 0.334975 & 0.529595 \\
    204   & 408   & 0.979988 & 36    & 0.235813 & 0.088235 & 0.333333 & 0.534452 \\
    205   & 410   & 0.980036 & 36    & 0.238273 & 0.087805 & 0.331707 & 0.539254 \\
    206   & 412   & 0.980081 & 36    & 0.240735 & 0.087379 & 0.332524 & 0.534462 \\
    207   & 414   & 0.980124 & 36    & 0.243196 & 0.086957 & 0.330918 & 0.539217 \\
    208   & 416   & 0.980165 & 36    & 0.245659 & 0.086538 & 0.329327 & 0.543919 \\
    209   & 418   & 0.980203 & 36    & 0.248122 & 0.086124 & 0.330144 & 0.539158 \\
    210   & 420   & 0.980239 & 36    & 0.250586 & 0.085714 & 0.328571 & 0.543815 \\
    211   & 422   & 0.980273 & 36    & 0.25305 & 0.085308 & 0.327014 & 0.548421 \\
    212   & 424   & 0.980305 & 36    & 0.255513 & 0.084906 & 0.32783 & 0.543692 \\
    213   & 426   & 0.980337 & 37    & 0.230127 & 0.086854 & 0.326291 & 0.548254 \\
    214   & 428   & 0.980384 & 37    & 0.232454 & 0.086449 & 0.324766 & 0.552767 \\
    215   & 430   & 0.98043 & 37    & 0.234782 & 0.086047 & 0.325581 & 0.548069 \\
    216   & 432   & 0.980473 & 37    & 0.237111 & 0.085648 & 0.324074 & 0.552541 \\
    217   & 434   & 0.980514 & 37    & 0.239441 & 0.085253 & 0.322581 & 0.556963 \\
    218   & 436   & 0.980553 & 37    & 0.241771 & 0.084862 & 0.323394 & 0.552298 \\
    219   & 438   & 0.980591 & 37    & 0.244103 & 0.084475 & 0.321918 & 0.55668 \\
    220   & 440   & 0.980626 & 37    & 0.246434 & 0.084091 & 0.320455 & 0.561015 \\
    221   & 442   & 0.980659 & 37    & 0.248767 & 0.08371 & 0.321267 & 0.556383 \\
    222   & 444   & 0.980691 & 37    & 0.251099 & 0.083333 & 0.31982 & 0.56068 \\
    223   & 446   & 0.98072 & 37    & 0.253432 & 0.08296 & 0.318386 & 0.56493 \\
    224   & 448   & 0.980754 & 38    & 0.228757 & 0.084821 & 0.319196 & 0.56033 \\
    225   & 450   & 0.980798 & 38    & 0.230962 & 0.084444 & 0.317778 & 0.564545 \\
    226   & 452   & 0.98084 & 38    & 0.233169 & 0.084071 & 0.316372 & 0.568714 \\
    227   & 454   & 0.98088 & 38    & 0.235376 & 0.0837 & 0.317181 & 0.564146 \\
    228   & 456   & 0.980918 & 38    & 0.237585 & 0.083333 & 0.315789 & 0.568281 \\
    229   & 458   & 0.980954 & 38    & 0.239794 & 0.082969 & 0.31441 & 0.572371 \\
    230   & 460   & 0.980989 & 38    & 0.242004 & 0.082609 & 0.315217 & 0.567836 \\
    231   & 462   & 0.981022 & 38    & 0.244215 & 0.082251 & 0.313853 & 0.571893 \\
    232   & 464   & 0.981053 & 38    & 0.246426 & 0.081897 & 0.3125 & 0.575907 \\
    233   & 466   & 0.981083 & 38    & 0.248637 & 0.081545 & 0.311159 & 0.579877 \\
    234   & 468   & 0.981111 & 38    & 0.250849 & 0.081197 & 0.311966 & 0.575386 \\
    235   & 470   & 0.981142 & 39    & 0.226879 & 0.082979 & 0.310638 & 0.579326 \\
    236   & 472   & 0.981183 & 39    & 0.228972 & 0.082627 & 0.309322 & 0.583224 \\
    237   & 474   & 0.981222 & 39    & 0.231066 & 0.082278 & 0.310127 & 0.578766 \\
    238   & 476   & 0.98126 & 39    & 0.233162 & 0.081933 & 0.308824 & 0.582634 \\
    239   & 478   & 0.981296 & 39    & 0.235258 & 0.08159 & 0.307531 & 0.586462 \\
    240   & 480   & 0.98133 & 39    & 0.237355 & 0.08125 & 0.308333 & 0.582036 \\
    241   & 482   & 0.981363 & 39    & 0.239452 & 0.080913 & 0.307054 & 0.585835 \\
    242   & 484   & 0.981394 & 39    & 0.241551 & 0.080579 & 0.305785 & 0.589594 \\
    243   & 486   & 0.981424 & 39    & 0.24365 & 0.080247 & 0.306584 & 0.585201 \\
    244   & 488   & 0.981452 & 39    & 0.245749 & 0.079918 & 0.305328 & 0.588933 \\
    245   & 490   & 0.981478 & 39    & 0.247849 & 0.079592 & 0.304082 & 0.592626 \\
    246   & 492   & 0.981505 & 40    & 0.224576 & 0.081301 & 0.304878 & 0.588265 \\
    247   & 494   & 0.981543 & 40    & 0.226564 & 0.080972 & 0.303644 & 0.591932 \\
    248   & 496   & 0.98158 & 40    & 0.228554 & 0.080645 & 0.302419 & 0.595561 \\
    249   & 498   & 0.981616 & 40    & 0.230545 & 0.080321 & 0.301205 & 0.599153 \\
    250   & 500   & 0.98165 & 40    & 0.232537 & 0.08  & 0.302 & 0.594836 \\
    251   & 502   & 0.981683 & 40    & 0.234529 & 0.079681 & 0.300797 & 0.598403 \\
    252   & 504   & 0.981714 & 40    & 0.236523 & 0.079365 & 0.299603 & 0.601933 \\
    253   & 506   & 0.981744 & 40    & 0.238517 & 0.079051 & 0.300395 & 0.597648 \\
    254   & 508   & 0.981772 & 40    & 0.240512 & 0.07874 & 0.299213 & 0.601156 \\
    255   & 510   & 0.9818 & 40    & 0.242507 & 0.078431 & 0.298039 & 0.604627 \\
    256   & 512   & 0.981825 & 40    & 0.244503 & 0.078125 & 0.298828 & 0.600373 \\
    257   & 514   & 0.98185 & 40    & 0.246499 & 0.077821 & 0.297665 & 0.603822 \\
    258   & 516   & 0.981881 & 41    & 0.223807 & 0.079457 & 0.296512 & 0.607236 \\
    259   & 518   & 0.981916 & 41    & 0.225699 & 0.079151 & 0.297297 & 0.603012 \\
    260   & 520   & 0.98195 & 41    & 0.227593 & 0.078846 & 0.296154 & 0.606405 \\
    261   & 522   & 0.981983 & 41    & 0.229488 & 0.078544 & 0.295019 & 0.609763 \\
    262   & 524   & 0.982014 & 41    & 0.231383 & 0.078244 & 0.293893 & 0.613087 \\
    263   & 526   & 0.982045 & 41    & 0.23328 & 0.077947 & 0.294677 & 0.608909 \\
    264   & 528   & 0.982074 & 41    & 0.235177 & 0.077652 & 0.293561 & 0.612213 \\
    265   & 530   & 0.982101 & 41    & 0.237075 & 0.077358 & 0.292453 & 0.615484 \\
    266   & 532   & 0.982128 & 41    & 0.238973 & 0.077068 & 0.293233 & 0.611335 \\
    267   & 534   & 0.982153 & 41    & 0.240872 & 0.076779 & 0.292135 & 0.614587 \\
    268   & 536   & 0.982177 & 41    & 0.242772 & 0.076493 & 0.291045 & 0.617806 \\
    269   & 538   & 0.982199 & 41    & 0.244672 & 0.076208 & 0.289963 & 0.620993 \\
    270   & 540   & 0.982232 & 42    & 0.22256 & 0.077778 & 0.290741 & 0.616889 \\
    271   & 542   & 0.982265 & 42    & 0.224363 & 0.077491 & 0.289668 & 0.620058 \\
    272   & 544   & 0.982296 & 42    & 0.226167 & 0.077206 & 0.288603 & 0.623196 \\
    273   & 546   & 0.982327 & 42    & 0.227973 & 0.076923 & 0.289377 & 0.619121 \\
    274   & 548   & 0.982356 & 42    & 0.229779 & 0.076642 & 0.288321 & 0.622241 \\
    275   & 550   & 0.982385 & 42    & 0.231585 & 0.076364 & 0.287273 & 0.625331 \\
    276   & 552   & 0.982412 & 42    & 0.233393 & 0.076087 & 0.288043 & 0.621285 \\
    277   & 554   & 0.982438 & 42    & 0.235201 & 0.075812 & 0.287004 & 0.624358 \\
    278   & 556   & 0.982462 & 42    & 0.23701 & 0.07554 & 0.285971 & 0.627402 \\
    279   & 558   & 0.982486 & 42    & 0.238819 & 0.075269 & 0.284946 & 0.630415 \\
    280   & 560   & 0.982509 & 42    & 0.240629 & 0.075 & 0.285714 & 0.626412 \\
    281   & 562   & 0.98253 & 42    & 0.242439 & 0.074733 & 0.284698 & 0.62941 \\
    282   & 564   & 0.982561 & 43    & 0.220904 & 0.076241 & 0.283688 & 0.632379 \\
    283   & 566   & 0.982592 & 43    & 0.222624 & 0.075972 & 0.284452 & 0.628404 \\
    284   & 568   & 0.982621 & 43    & 0.224345 & 0.075704 & 0.283451 & 0.631358 \\
    285   & 570   & 0.98265 & 43    & 0.226066 & 0.075439 & 0.282456 & 0.634284 \\
    286   & 572   & 0.982678 & 43    & 0.227788 & 0.075175 & 0.281469 & 0.637181 \\
    287   & 574   & 0.982705 & 43    & 0.229511 & 0.074913 & 0.28223 & 0.633249 \\
    288   & 576   & 0.98273 & 43    & 0.231235 & 0.074653 & 0.28125 & 0.636132 \\
    289   & 578   & 0.982755 & 43    & 0.23296 & 0.074394 & 0.280277 & 0.638988 \\
    290   & 580   & 0.982778 & 43    & 0.234685 & 0.074138 & 0.281034 & 0.635083 \\
    291   & 582   & 0.982801 & 43    & 0.236411 & 0.073883 & 0.280069 & 0.637926 \\
    292   & 584   & 0.982823 & 43    & 0.238137 & 0.07363 & 0.27911 & 0.640741 \\
    293   & 586   & 0.982843 & 43    & 0.239864 & 0.073379 & 0.278157 & 0.64353 \\
    294   & 588   & 0.98287 & 44    & 0.218899 & 0.07483 & 0.278912 & 0.639666 \\
    295   & 590   & 0.982899 & 44    & 0.220541 & 0.074576 & 0.277966 & 0.642442 \\
    296   & 592   & 0.982927 & 44    & 0.222183 & 0.074324 & 0.277027 & 0.645192 \\
    297   & 594   & 0.982955 & 44    & 0.223826 & 0.074074 & 0.277778 & 0.641356 \\
    298   & 596   & 0.982981 & 44    & 0.22547 & 0.073826 & 0.276846 & 0.644094 \\
    299   & 598   & 0.983007 & 44    & 0.227115 & 0.073579 & 0.27592 & 0.646806 \\
    300   & 600   & 0.983031 & 44    & 0.228761 & 0.073333 & 0.275 & 0.649493 \\
 \end{supertabular} }%
\label{tab:newtable}
\end{center}


\begin{thebibliography}{9}
\bibitem{Andre} D. Andr\'e (1887). ``Solution directe du probl\`eme resolu par 
M.\ Bertrand.'' 
\emph{Comptes Randus Acad.\ Sci.} Paris \textbf{105}, 436--437.
\bibitem{Bachelier01} L.\ Bachelier (1901). ``Th\'eorie math\'ematique du
jeu.'' {\em Ann.\ Sci.\ Ecole Nat.\ Sup.}, 3$^e$ s\'er., {\textbf 18}, 143-209.
\bibitem{Bachelier12} L.\ Bachelier (1912)$^*$. \emph{Calcul des probabilit\'es
}, \textbf{1}. Gauthier--Villars, Paris. 
\bibitem{Bachelier39} L.\ Bachelier (1939). \emph{Les nouvelles m\'ethodes du
calcul des probabilit\'es}. Gauthier--Villars, Paris.  
\bibitem{Blackman56} J.\ Blackman (1956). An extension of the Kolmogorov
distribution. \emph{Ann.\ Math.\ Statist.} \textbf{27}, 513--520. [cf.\ 
\cite{Blackman58}]
\bibitem{Blackman58} 
J.\ Blackman (1958). ``Correction to `An extension of the Kolmogorov 
distribution.'\,'' \emph{Ann.\ Math.\ Statist.} \textbf{29}, 318--322.

\bibitem{DKW} A. Dvoretzky, J. Kiefer, J. Wolfowitz (1956). ``Asymptotic 
minimax character of the sample distribution function and of the classical 
multinomial estimator.'' \emph{Ann.\ Math.\ Statist.} 
\textbf{27},  642--669.
\bibitem{Dwass} M. Dwass (1967). ``Simple random walk and rank order 
statistics.'' \emph{Ann. Math. Statist.} {\textbf{38}}, 1042--1053.
\bibitem{GK} B. V. Gnedenko, V.\ S.\ Korolyuk (1951). ``On the maximum 
discrepancy between two empirical distributions.'' \emph{Dokl.\ Akad.\ Nauk 
SSSR} {\textbf{80}}, 525--528 [Russian]; \emph{Sel.\ Transl.\ Math.\ Statist.\
Probab.} {\textbf{1}} (1961), 13--16.

\bibitem{Hodges} J.\ L.\ Hodges (1957). ``The significance probability
of the Smirnov two sample test.'' {\em Arkiv f\"or Matematik} {\bf 3},
469-486.

\bibitem{Kimjennrich} P.\ J.\ Kim and R.\ I.\ Jennrich (1970), ``Tables of
the exact sampling distribution of the two-sample Kolmogorov--Smirnov
criterion, $D_{mn}$, $m\leq n$, in {\em Selected Tables in Mathematical
Statistics}, Institute of Mathematical Statistics, ed.\ H.\ L.\ Harter
and D.\ B.\ Owen; Repub.\ Markham, Chicago, 1973.

\bibitem{Massart2} P.\ Massart (1990). ``The tight constant in the 
Dvoretzky--Kiefer--Wolfowitz inequality.'' \emph{Ann. Probability} 
\textbf{18}, 1269--1283.

\bibitem{ErrorTerm} T.\ S.\ Nanjundiah (1959). ``Note on Stirling's Formula.'' 
\emph{Amer. Math. Monthly.} \textbf{66}, 701--703.


\bibitem{recipes} W.\ H.\ Press, S.\ A.\ Teukolsky, W.\ T.\ Vetterling,
B.\ P.\ Flannery (1992). \emph{Numerical Recipes in FORTRAN: The Art of
Scientific Computing}, 2d.\ ed., Cambridge University Press.

\bibitem{Stephens70} M.\ A.\ Stephens (1970). Use of the
Kolmogorov--Smirnov, Cram\'er--Von Mises and related statistics 
without extensive tables. \emph{J.\ Royal Statistical Soc.\ Ser.\ B
(Methodological)} {\bf 32}, 115-122. 

\bibitem{FWRDshort} F.\ Wei and R.\ M.\ Dudley (2011). 
``Two-sample Dvoretzky--Kiefer--Wolfowitz Inequalities.''
Preprint.

\

\noindent
${\ }^*$  An asterisk indicates items of which we learned from secondary
sources but which we have not seen in the original.
 
\end{thebibliography}
\end{document}